\documentclass[11pt]{article}

\title{Generating the Johnson filtration}
\author{Thomas Church\thanks{Supported in part by NSF grants DMS-1103807 and DMS-1350138}\ \ and Andrew Putman\thanks{Supported in part by NSF grant DMS-1255350 and the Alfred P.\ Sloan Foundation}\vspace{-6pt}}
\date{}

\usepackage{amsmath,amssymb,amsthm,amscd,amsfonts}
\usepackage{epsfig,pinlabel}
\usepackage[hmargin=1.25in, vmargin=1in]{geometry}
\usepackage[font=small,format=plain,labelfont=bf,up,textfont=it,up]{caption}
\usepackage{type1cm}
\usepackage{calc}
\usepackage{bm}
\usepackage{microtype}
\usepackage{url}
\usepackage[all,cmtip]{xy}

\usepackage[dvipdfm,bookmarks, bookmarksdepth=2,
colorlinks=true, linkcolor=black, citecolor=black, urlcolor=black]{hyperref}

\usepackage{paralist}

\theoremstyle{plain}
\newtheorem{theorem}{Theorem}[section]
\newtheorem{maintheorem}{Theorem}
\newtheorem*{boundedgentheorem}{Theorem~\ref{maintheorem:boundedgeneration}}
\newtheorem{proposition}[theorem]{Proposition}
\newtheorem{lemma}[theorem]{Lemma}

\newtheorem{claims}{Claim}

\newcommand\BeginClaims{\setcounter{claims}{0}}

\theoremstyle{definition}

\newtheorem{definition}[theorem]{Definition}

\theoremstyle{remark}
\newtheorem{remark}[theorem]{Remark}

\DeclareMathOperator{\Hom}{Hom}

\DeclareMathOperator{\Image}{Im}

\DeclareMathOperator{\Mod}{Mod}
\DeclareMathOperator{\AF}{AF}

\newcommand\Torelli{\ensuremath{{\mathcal I}}}
\DeclareMathOperator{\IA}{IA}
\DeclareMathOperator{\Sp}{Sp}
\DeclareMathOperator{\GL}{GL}
\DeclareMathOperator{\SL}{SL}

\newcommand\SpLie{\ensuremath{\mathfrak{sp}}}
\newcommand\GLLie{\ensuremath{\mathfrak{gl}}}

\newcommand\Z{\ensuremath{\mathbb{Z}}}

\newcommand\N{\ensuremath{\mathbb{N}}}

\DeclareMathOperator{\HH}{H}

\DeclareMathOperator{\Aut}{Aut}

\DeclareMathOperator{\Interior}{Int}

\newcommand\Set[2]{\ensuremath{\{\text{#1 $|$ #2}\}}}

\newcommand\Figure[4]{
\begin{figure}[t]
\centering
\centerline{\psfig{file=#2,scale=#4}}
\caption{#3}
\label{#1}
\end{figure}}

\newcommand\ialg{\ensuremath{\widehat{i}}}
\newcommand\FI{\ensuremath{\text{\tt FI}}}
\newcommand\Grp{\ensuremath{\text{\tt Grp}}}
\newcommand\CGrp{\ensuremath{\text{\tt CGrp}}}

\newcommand\One{\ensuremath{\mathbb{I}}}
\newcommand\Zero{\ensuremath{\mathbb{O}}}

\newcommand{\Stab}[2]{\ensuremath{\mathcal{C}(#1,#2)}}

\newcommand{\para}[1]{\medskip\noindent\textbf{#1.}}
\newcommand{\into}{\hookrightarrow}
\newcommand{\onto}{\twoheadrightarrow}
\newcommand{\bwedge}{\textstyle{\bigwedge}}
\newcommand{\abs}[1]{\left\lvert#1\right\rvert}
\newcommand{\normal}{\lhd}

\DeclareMathOperator{\id}{id}
\DeclareMathOperator{\ab}{ab}
\newcommand{\coloneq}{\mathrel{\mathop:}\mkern-1.2mu=}
\newcommand{\Sym}{\ensuremath{\mathfrak{S}}}
\newcommand{\YY}{Y}
\newcommand{\ZZ}{Z}
\newcommand{\Lie}{\ensuremath{\mathcal{L}}}
\DeclareMathOperator{\Der}{Der}
\DeclareMathOperator{\InnerAut}{InnerAut}
\DeclareMathOperator{\InnerDer}{InnerDer}
\DeclareMathOperator{\gr}{gr}
\DeclareMathOperator{\PP}{PP}

\newcommand{\arXiv}[1]{\href{http://arxiv.org/abs/#1}{arXiv:#1}}
\newcommand{\myemail}[1]{\href{mailto:#1}{\nolinkurl{#1}}}

\begin{document}

\maketitle

\vspace{-20pt}
\begin{abstract}
For $k \geq 1$, let $\Torelli_g^1(k)$ be the $k^{\text{th}}$ term in the Johnson filtration
of the mapping class group of a genus $g$ surface with one boundary component.  
We prove that for all $k \geq 1$, there exists some $G_k\geq 0$ such that $\Torelli_g^1(k)$ is generated
by elements which are supported on subsurfaces whose genus is at most $G_k$.  We also prove
similar theorems for the Johnson filtration of $\Aut(F_n)$ and
for certain mod-$p$ analogues of the Johnson filtrations of both the mapping class group and
of $\Aut(F_n)$.  The main tools used in the proofs are the related theories of FI-modules (due to the
first author with Ellenberg and Farb) and central stability (due to the second author), both
of which concern the representation theory of the symmetric groups over $\Z$.
\end{abstract}

\section{Introduction}

In this paper, we use techniques from representation theory to prove that the terms of the Johnson filtrations
of both the mapping class group and the automorphism group of a free group are generated by elements
whose complexity is bounded in a sense to be made precise below.

\para{Mapping class group} Let $\Sigma_g^k$ denote a compact oriented genus $g$ surface with $k$ boundary components. 
Let $\Mod_g^1$ be the \emph{mapping class group} of $\Sigma_g^1$, i.e.\ the group of isotopy classes of orientation-preserving
homeomorphisms of $\Sigma_g^1$ that restrict to the identity on $\partial \Sigma_g^1$.  

Choosing a basepoint
$\ast \in \partial \Sigma_g^1$, the group $\Mod_g^1$ acts on $\pi\coloneq \pi_1(\Sigma_g^1,\ast)$.  For a group $G$,
let $\gamma_k(G)$ be the $k^{\text{th}}$ term in the lower central series of $G$, so $\gamma_1(G) = G$
and $\gamma_{k+1}(G) = [\gamma_k(G),G]$ for $k \geq 1$.
The action of $\Mod_g^1$ on $\pi$ preserves $\gamma_k(\pi)$, so there is
an induced action of $\Mod_g^1$ on $\pi / \gamma_k(\pi)$.  The $k^{\text{th}}$ term of the \emph{Johnson
filtration} of $\Mod_g^1$, denoted $\Torelli_g^1(k)$, is the kernel of the action of $\Mod_g^1$
on $\pi / \gamma_{k+1}(\pi)$.  The Johnson filtration was defined by Johnson in \cite{JohnsonSurvey} and 
has connections to number theory (see Matsumoto~\cite{Matsumoto})
and $3$-manifolds (see Garoufalidis--Levine~\cite{GaroufalidisLevine}); however, many basic questions about it remain open.

\para{Generators in low degree}
Let $T_x \in \Mod_g^1$ denote the Dehn twist about a simple closed curve $x$ on $\Sigma_g^1$. It was proved independently by Lickorish~\cite{LickorishTwists} and Mumford~\cite{MumfordAbelianQuotients}, building on the work of Dehn, that $\Mod_g^1$ is generated by Dehn twists about nonseparating simple closed curves.

Let $\Torelli_g^1 \coloneq \Torelli_g^1(1)$.  The group $\Torelli_g^1$ is known as the 
\emph{Torelli group}; it is the kernel of the action
of $\Mod_g^1$ on $\pi / \gamma_2(\pi) \cong \HH_1(\Sigma_g^1;\Z)$.
A \emph{genus $\ell$ bounding pair map} is a product $T_y T_z^{-1}$, where $y$ and $z$
are disjoint nonseparating simple closed curves on $\Sigma_g^1$ whose union $y \cup z$ separates
$\Sigma_g^1$ into two subsurfaces, one homeomorphic to $\Sigma_\ell^2$ and the other to $\Sigma_{g-\ell-1}^3$
(see Figure~\ref{figure:torelligen}).
Making essential use of work of Powell~\cite{PowellTorelli}, Johnson~\cite{JohnsonFirst} proved that $\Torelli_g^1$
is generated by genus $1$ bounding pair maps for $g \geq 3$.  
See \cite{PutmanCutPasteTorelli} and Hatcher--Margalit \cite{HatcherMargalit} for modern proofs of the necessary results
of Powell.

\Figure{figure:torelligen}{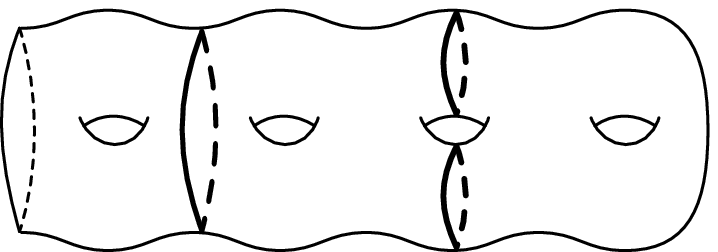}{A genus $3$ separating twist $T_x$ and a genus $1$ bounding
pair map $T_y T_z^{-1}$.}{115}

The group $\Torelli_g^1(2)$ is known as the \emph{Johnson kernel}.  A \emph{genus $\ell$ separating twist}
is a mapping class $T_x$, where $x$ is a simple closed curve that separates $\Sigma_g^1$ into two
subsurfaces, one homeomorphic to $\Sigma_{\ell}^1$ and the other to $\Sigma^2_{g-\ell}$
(see Figure~\ref{figure:torelligen}).
Johnson~\cite{JohnsonII} proved that $\Torelli_g^1(2)$ is generated by genus $1$ and $2$ separating twists. 

\para{Higher degree}
For $k \geq 3$, no interesting generating set for $\Torelli_g^1(k)$ is known (of course, one could do uninteresting 
things like
taking the entire group as a generating set).  An appealing feature of the generating sets above is that the generators are ``simple'', in the sense that 
they are supported on small 
subsurfaces (i.e.\ subsurfaces with 1 boundary component and bounded genus).  Our first main theorem says that for every $k\geq 1$ the group
$\Torelli_g^1(k)$ can be generated by elements supported on subsurfaces of uniformly-bounded size.

In fact, we can be somewhat more precise.  Fix a symplectic basis $\{a_1,b_1,\ldots,a_g,b_g\}$
for $\HH_1(\Sigma_g^1;\Z) \cong \Z^{2g}$, i.e.\ a free basis such that
\[\ialg(a_i,a_j) = \ialg(b_i,b_j) = 0 \quad \text{and} \quad \ialg(a_i,b_j) = \delta_{ij},\]% \quad \quad (1 \leq i,j \leq g),\]
where $\ialg(\cdot,\cdot)$ is the algebraic intersection pairing. Say that a subsurface $S$ of $\Sigma_g^1$
is \emph{homologically standard} if $S$ has one boundary component and the
image of $\HH_1(S;\Z)$ in $\HH_1(\Sigma_g^1;\Z)$ is $\langle a_{i}, b_{i} \text{ $|$ } i \in I \rangle$
for some $I \subset \{1,\ldots,g\}$.  Our theorem is then as follows.

\begin{maintheorem}[Generators for Johnson filtration]
\label{maintheorem:modgentorelli}
For every $k \geq 1$, there exists some $G_k \geq 0$ such that for all $g \geq 1$, the group $\Torelli_g^1(k)$ is generated
by elements which are supported on homologically standard subsurfaces of $\Sigma_g^1$ whose
genus is at most $G_k$.
\end{maintheorem}

\begin{remark}
We emphasize that the constant $G_k$ in Theorem~\ref{maintheorem:modgentorelli} 
depends only on $k$, not on $g$.  Otherwise, the theorem would be rather trivial!
\end{remark}

\noindent
Somewhat surprisingly, our proof of Theorem~\ref{maintheorem:modgentorelli} is purely an existence proof; it gives no information about how large the constants $G_k$ must be.  The following theorem, however, implies that the bounds $G_k$ must tend to infinity.

\begin{maintheorem}[Lower bound on genus]
\label{maintheorem:modnongentorelli}
For all $k \geq 1$ and $g > k$, the group $\Torelli_g^1(k)$ is not generated by elements supported on subsurfaces with one boundary component and genus less than $\frac{k}{2}$.
\end{maintheorem}

\para{Automorphism groups of free groups}
The Johnson filtration can also be defined on the automorphism group $\Aut(F_n)$ of the free group $F_n=\langle x_1,\ldots,x_n\rangle$.
Let $\IA_n(k)$ denote the kernel of the action of $\Aut(F_n)$ on $F_n / \gamma_{k+1}(F_n)$.  The group $\IA_n \coloneq \IA_n(1)$
consists of automorphisms in  $\Aut(F_n)$ acting trivially on $F_n / \gamma_2(F_n) \cong \Z^n$, and is often known as the Torelli subgroup
of $\Aut(F_n)$. Magnus found a finite generating set for $\IA_n$ consisting of the following two types of elements.
\begin{compactitem}
\item For distinct $1 \leq i,j \leq n$, let $c_{ij} \in \IA_n$ be the automorphism defined by
\[c_{ij}(x_\ell) = \begin{cases}
x_j^{-1} x_\ell x_j & \text{if $\ell=i$},\\
x_\ell & \text{otherwise}.
\end{cases}\]
\item For distinct $1 \leq i,j,k \leq n$, let $m_{ijk} \in \IA_n$ be the automorphism defined by
\[m_{ijk}(x_{\ell}) = \begin{cases}
x_{\ell} [x_j,x_k] & \text{if $\ell=i$},\\
x_{\ell} & \text{otherwise}.
\end{cases}\]
\end{compactitem}
Magnus \cite{MagnusGenerators} proved that $\IA_n$ is generated by the automorphisms $c_{ij}$ and $m_{ijk}$; see Bestvina--Bux--Margalit~\cite{BestvinaBuxMargalitIA} and Day--Putman~\cite{DayPutmanGen} for modern proofs of Magnus's theorem.
For $k\geq 2$, a generating set for $\IA_n(k)$ is not known.

\para{Subsurfaces for free groups}
To state a version of Theorem~\ref{maintheorem:modgentorelli} for $\IA_n(k)$, we need an appropriate analogue of
``supported on a subsurface'' for $\Aut(F_n)$.  A \emph{splitting} of $F_n$ consists of subgroups
$A,B < F_n$ such that $F_n$ splits as the free product $F_n = A \ast B$.  The \emph{rank} of a splitting $A \ast B$
is the rank of the free group $A$ (notice that this is different from the rank of the splitting $B \ast A$).
We will say that an element $\varphi \in \Aut(F_n)$ is \emph{supported}
on a splitting $A \ast B$ if $\varphi(A) = A$ and $\varphi|_B = \text{id}$.  
For example, Magnus's generator $c_{ij}$ is supported on a splitting of
rank $2$, and $m_{ijk}$ is supported on a splitting of rank $3$.  
We will prove that for all $k \geq 1$, the group $\IA_n(k)$ is generated by elements supported on
splittings whose rank is uniformly bounded.

Just as for the mapping class group, we will actually prove something a bit more precise.
Let $\{e_1,\ldots,e_n\}$ be the standard basis for $F_n^{\ab}\cong \Z^n$. Say that a splitting 
$A \ast B$ of $F_n$ is \emph{homologically standard} if there is some $I \subset \{1,\ldots,n\}$ such 
that the images of $A$ and $B$ in $F_n^{\ab}$ are $A^{\ab}= \langle e_i \,|\, i \in I \rangle$ 
and $B^{\ab}= \langle e_i \,|\, i \notin I \rangle$.
We then have the following theorem.

\begin{maintheorem}[Generators for Johnson filtration of $\Aut(F_n)$]
\label{maintheorem:autgentorelli}
For every $k \geq 1$, there exists some $N_k \geq 0$ such that for  all $n \geq 1$, the group $\IA_n(k)$ is generated
by elements which are supported on homologically standard splittings whose
rank is at most $N_k$.
\end{maintheorem}

\noindent
We will also prove the following analogue of Theorem~\ref{maintheorem:modnongentorelli}.

\begin{maintheorem}[Lower bound on rank]
\label{maintheorem:autnongentorelli}
For all $k \geq 1$ and $n > k$, the group $\IA_n(k)$ is not generated by elements supported on splittings
of rank less than $k$.
\end{maintheorem}

\para{Mod-$\bm{p}$ lower central series}
Fix a prime $p$.  In recent work \cite{CooperThesis}, Cooper has introduced two mod-$p$ analogues of 
the Johnson filtration.  The starting points are two different mod-$p$ analogues of the lower
central series of a group $G$.  If $G'$ is a subgroup of $G$ and $\ell \geq 1$, then denote by $(G')^{\ell}$ the subgroup
of $G$ generated by $\Set{$x^{\ell}$}{$x \in G'$}$.
\begin{compactitem}
\item The \emph{mod-$p$ Stallings filtration} of $G$ is the inductively defined filtration
\[\gamma_1^S(G) = G \quad \quad \text{and} \quad \quad \gamma_{k+1}^S(G) = [\gamma_k^S(G),G] \cdot (\gamma_k^S(G))^p \text{ for $k \geq 1$}.\]
This filtration first appeared in Stallings~\cite{Stallings}.
\item The \emph{mod-$p$ Zassenhaus filtration} of $G$ is defined in terms of the usual lower
central series via the formula
\[\gamma_k^Z(G) = \prod_{i p^j \geq k} (\gamma_i(G))^{p^j}.\]
This filtration first appeared in Zassenhaus~\cite{Zassenhaus}.
\end{compactitem}
If $G$ is finitely generated, the quotients $G / \gamma_k^S(G)$ and $G / \gamma_k^Z(G)$ are both
finite $p$-groups.  We have
\[G / \gamma_2^Z(G) \cong G / \gamma_2^Z(G) \cong \HH_1(G;\Z/p);\]
however, for $k \geq 3$ these two filtrations differ.  

\para{Mod-$\bm{p}$ Johnson filtrations}
We define $\Torelli_g^{1,S}(k)$
and $\Torelli_g^{1,Z}(k)$ to be the kernels of the actions of $\Mod_g^1$ on
$\pi / \gamma_{k+1}^S(\pi)$ and $\pi/\gamma_{k+1}^Z(\pi)$, respectively.  Observe that
both $\Torelli_g^{1,Z}(1)$ and $\Torelli_g^{1,S}(1)$ coincide with the \emph{level-$p$
congruence subgroup} $\Mod_g^1(p)$, that is, the kernel of the action
of $\Mod_g^1$ on $\HH_1(\Sigma_g^1;\Z/p)$.  All the groups $\Torelli_g^{1,S}(k)$ and $\Torelli_g^{1,Z}(k)$  in these filtrations are
finite-index subgroups of $\Mod_g^1$.

Similarly, 
we define $\IA_n^S(k)$ and $\IA_n^Z(k)$ to be the kernels of the actions of $\Aut(F_n)$
on $F_n / \gamma_{k+1}^S(F_n)$ and $F_n / \gamma_{k+1}^Z(F_n)$, respectively.  Both
$\IA_n^S(1)$ and $\IA_n^Z(1)$ coincide with the \emph{level-$p$ congruence subgroup}
$\Aut(F_n,p)$, that is, the kernel of the action of $\Aut(F_n)$ on
$\HH_1(F_n;\Z/p)\cong (\Z/p)^n$.  Again, all of the terms in these filtrations are finite-index subgroups
of $\Aut(F_n)$.

\begin{remark}
Yet another mod-$p$ Johnson filtration was defined by Perron in \cite{Perron} using
the Fox calculus, but Cooper \cite{CooperThesis} proved that Perron's filtration equals
the Zassenhaus filtration.
\end{remark}

\para{Generators for mod-$\bm{p}$ Johnson filtrations}
Cooper \cite{CooperThesis} proved many interesting results about these filtrations.  In
particular, he found simple generating sets for $\Torelli_g^{1,S}(k)$ and
$\Torelli_g^{1,Z}(k)$ for $k=1$ and $k=2$.  We are able to prove analogues 
of Theorems~\ref{maintheorem:modgentorelli} and \ref{maintheorem:autgentorelli} for these filtrations.
Let $\{a_1,b_1,\ldots,a_g,b_g\}$ the standard symplectic basis
for $\HH_1(\Sigma_g^1;\Z/p)$.  Say that a subsurface $S$ of $\Sigma_g^1$
is \emph{$\Z/p$-homologically standard} if $S$ has one boundary component and the
image of $\HH_1(S;\Z/p)$ in $\HH_1(\Sigma_g^1;\Z/p)$ is $\langle a_{i}, b_{i} \,|\, i \in I \rangle$
for some $I \subset \{1,\ldots,g\}$.  We then have the following.

\begin{maintheorem}[Generators for mod-$p$ Johnson filtrations]
\label{maintheorem:modgenlevel}
%Let $\star$ be either $S$ or $Z$, and f
Fix a prime $p$.  For all $k \geq 1$, there exists some $G_k \geq 0$ (depending
on %$\star$ and
$p$) 
such that for  all $g \geq 1$, both $\Torelli_g^{1,S}(k)$ and $\Torelli_g^{1,Z}(k)$ are generated
by elements which are supported on a $\Z/p$-homologically standard subsurface of $\Sigma_g^1$ of genus $\leq G_k$.
\end{maintheorem}

Similarly, let $\{e_1,\ldots,e_n\}$ be the standard basis for $\HH_1(F_n;\Z/p)\cong (\Z/p)^n$. Say that a splitting
$A \ast B$ of $F_n$ is \emph{$\Z/p$-homologically standard} if for some $I \subset \{1,\ldots,n\}$,  the images of $A$ and $B$ in $\HH_1(F_n;\Z/p)$ are $\HH_1(A;\Z/p)=\langle e_i \,|\, i \in I \rangle$
and $\HH_1(B;\Z/p)=\langle e_i \,|\, i \notin I \rangle$.
We then have the following.

\begin{maintheorem}[Generators for mod-$p$ Johnson filtrations of $\Aut(F_n)$]
\label{maintheorem:autgenlevel}
Fix a prime $p$.  For all $k \geq 1$, there exists some $N_k \geq 0$ (depending
on $p$)
such that for  all $n \geq 1$, both $\IA_n^{S}(k)$ and $\IA_n^Z(k)$ are generated
by elements which are supported on a $\Z/p$-homologically standard splitting
of $F_n$ of rank $\leq N_k$.
\end{maintheorem}

\para{Central stability}
Though our theorems concern topology and infinite group theory, the main tools
used in their proofs concern the representation theory of the symmetric group.  In particular,
we use the notion of \emph{central stability} for representations of the symmetric group, which
was introduced by the second author in \cite{PutmanRepStabilityCongruence} to study the homology
groups of congruence subgroups of $\GL_n(\Z)$.  Roughly speaking, this allows us to give an inductive
description of the images of the \emph{higher Johnson homomorphisms}, which are an important
sequence of abelian quotients of the terms of the Johnson filtrations.  The key advance that makes
this possible is the recent theorem of the first author with Ellenberg, Farb, and Nagpal~\cite{ChurchEllenbergFarbNagpal}, which establishes a Noetherian property for FI-modules over $\Z$.
This theorem allows one to prove that certain sequences of representations are centrally stable
almost for free (in particular, with no detailed understanding of their structure, which seems
quite hard to achieve for the images of the higher Johnson homomorphisms).

\para{FI-groups}
To formulate the technical framework for our arguments, we introduce  FI-groups and weak FI-groups. An 
FI-group $G$ consists of a group $G_I$ for each finite subset $I\subset \N$, together with homomorphisms $G_I\to G_J$ for each injection $I\into J$ satisfying some natural compatibility conditions (see Definition \ref{def:FIgroup} below).  A weak FI-group consists of similar data, except that for some of these homomorphisms, we require only that they be compatible up to conjugacy. The main technical result of the paper is the following theorem. The terms involved have not yet been defined; see \S\ref{section:fidefs} below for their definitions.
\begin{maintheorem}[Bounded generation for central filtrations]
\label{maintheorem:boundedgeneration}
Let $G$ be a weak FI-group with a central filtration $\{G(k)\}_{k=1}^\infty$ of finite rank. If $G$ is boundedly generated, then $G(k)\normal G$ is boundedly normally generated for every $k\geq 1$.
\end{maintheorem}

\para{Generating sets for Torelli}
To apply Theorem~\ref{maintheorem:boundedgeneration} to the Torelli groups $\Torelli_g^1$, we need a strengthening of a recent theorem of the second 
author \cite{PutmanSmallGenset} concerning generating
sets for $\Torelli_g^1$.  
Johnson~\cite{JohnsonFinite} proved  that the Torelli groups $\Torelli_g^1$ are finitely generated for $g\geq 3$ with a generating set whose size
is exponential in $g$.  Johnson~\cite{JohnsonAbel} also proved that the rank of the abelianization of $\Torelli_g^1$ 
is cubic in $g$, which gives a lower bound on the size of any generating set for $\Torelli_g^1$.

The second author's theorem \cite{PutmanSmallGenset} 
says that $\Torelli_g^1$ is generated by $57 \binom{g}{3} + 2g + 1$ elements for $g \geq 3$. 
What is important to us is not the size of his generating set per se, but rather the
fact that his generators are supported on fairly simple subsurfaces of $\Sigma_g^1$: each element is supported on a genus 3 subsurface with multiple boundary components. To prove Theorem~\ref{maintheorem:modgentorelli}, we need to improve this generating set slightly, so that the generators are supported on $\binom{g}{3}$ different genus 3 subsurfaces with only \emph{one} boundary component.  We refer to
Proposition~\ref{proposition:modgentorelli} below for a precise description of
our new generating set, but we point out the following corollary.  Denote
by $\Torelli_g$ the Torelli group on a closed genus $g$ surface.

\begin{maintheorem}
\label{maintheorem:smallgenset}
For all $g \geq 3$, the groups $\Torelli_g$ and $\Torelli_g^1$ are each generated by $42 \binom{g}{3}$ elements.
\end{maintheorem}

\para{Outline}
In \S\ref{section:fi}, we introduce FI-groups, weak FI-groups, and their central filtrations; the main
result of this section is Theorem~\ref{maintheorem:boundedgeneration}.  In \S\ref{section:aut}
we show how to apply this to the automorphism group of a free group and prove Theorem~\ref{maintheorem:autgentorelli}. In \S\ref{section:mod} we show how to apply this to the mapping class group and
prove Theorem~\ref{maintheorem:modgentorelli}; to do this, we first prove Theorem~\ref{maintheorem:smallgenset}.  Next, in \S\ref{section:level} we discuss how to modify
our proof of Theorems~\ref{maintheorem:autgentorelli} and \ref{maintheorem:modgentorelli} to prove
Theorems~\ref{maintheorem:modgenlevel} and \ref{maintheorem:autgenlevel}.  Finally, in \S\ref{section:bounds}
we prove Theorems~\ref{maintheorem:modnongentorelli} and \ref{maintheorem:autnongentorelli}.

\para{Acknowledgements}
We wish to thank the referee for their careful reading of our paper. We are grateful to Shigeyuki Morita for informing us of a mistake in an earlier version, and to Yiwei She for pointing out an elegant fix.

\section{FI-groups and their central filtrations}
\label{section:fi}

This section contains all our general results on FI-groups and weak FI-groups.  The key result is Theorem~\ref{maintheorem:boundedgeneration}, which
we will later apply to prove Theorems~\ref{maintheorem:modgentorelli}, \ref{maintheorem:autgentorelli},
\ref{maintheorem:modgenlevel}, and \ref{maintheorem:autgenlevel}.

We begin in \S\ref{section:fidefs}
with general definitions, including all the definitions that are used in the statement of Theorem~\ref{maintheorem:boundedgeneration}.  We then discuss some technical
results in \S\ref{section:technical}. In \S\ref{section:centralstability}, we describe the related theories of central stability and FI-modules as they will be used in this paper.  Finally, we prove
Theorem~\ref{maintheorem:boundedgeneration} in \S\ref{section:generalproof}.

\subsection{FI-groups and weak FI-groups}
\label{section:fidefs}

In this section, we introduce FI-groups, weak FI-groups, and central filtrations of weak FI-groups, leading up to the statement of the key Theorem~\ref{maintheorem:boundedgeneration}.

\para{FI-groups}
Let $\N$ be the set of natural numbers, and let $\FI$ be the category whose 
objects are finite subsets
of $\N$ and whose morphisms are injections. Let $\Grp$ be the category of
groups and homomorphisms.

\begin{definition}
\label{def:FIgroup}
An \emph{FI-group} is a functor from $\FI$ to $\Grp$.  In other words, an FI-group $G$ consists
of the following data.
\begin{compactenum}[(i)]
\item For each finite set $I \subset \N$, a group $G_I$.
\item For each injection $f\colon I \into J$ between finite sets $I,J \subset \N$, a homomorphism
$G_f\colon G_I \rightarrow G_J$.  These homomorphisms must satisfy the following compatibility conditions.
\begin{compactenum}[a.]
\item For all finite sets $I \subset \N$, we have $G_{\id_I} = \id$, where $\id_I\colon I \rightarrow I$ is
the identity.
\item For all finite sets $I,J,K \subset \N$ and all injections $f\colon I \into J$ and
$g\colon J \into K$, we have $G_{g \circ f} = G_g \circ G_f$.
\end{compactenum}
\end{compactenum}
A morphism $\Psi\colon G\to H$ of FI-groups is a natural transformation of functors. In other words, $\Psi$ consists of a homomorphism $\Psi_I\colon G_I\to H_I$ for each finite set $I\subset \N$, so that for every injection $f\colon I\into J$ between finite sets $I,J\subset \N$ the following diagram commutes:
\[\xymatrix{
G_I \ar[r]^{G_f} \ar[d]_{\Psi_I} & G_J \ar[d]^{\Psi_J} \\
H_I \ar[r]^{H_f}                 & H_J}\]
The morphism $\Psi$ is an isomorphism (resp.\ an injection, resp.\ a surjection) 
if $\Psi_I$ is an isomorphism (resp.\ an injection, resp.\ a surjection) for all finite sets $I \subset \N$.
\end{definition}

\begin{remark}
FI-groups (and the related notion of FI-modules; see \S\ref{section:technical} below) were originally 
defined by the first author  with Ellenberg and Farb in
\cite{ChurchEllenbergFarbFI}. The definitions in that paper were slightly different
from ours, in that in \cite{ChurchEllenbergFarbFI} the category $\FI$ had \emph{all} finite sets
as its objects; however, this larger category is equivalent to our category.  
\end{remark}

\begin{remark}
Let $[n]=\{1,\ldots,n\}$. For each bijection $\sigma\colon [n] \rightarrow [n]$, 
we have a homomorphism $G_\sigma\colon G_{[n]}\to G_{[n]}$.  Together 
these give an action of the symmetric group $\Sym_n$ on $G_{[n]}$.
\end{remark}

\para{Weak FI-groups}
In \S\ref{section:aut}, we will see that the automorphism groups of free groups can be naturally
viewed as an FI-group.
Unfortunately, the mapping class groups of surfaces do not form an FI-group.  However, they do satisfy a weaker form of functoriality that is sufficient for our purposes.  

If $A$ and $B$ are groups,
then $B$ acts by conjugation on the set of homomorphisms from $A$ to $B$.  A \emph{homomorphism-modulo-conjugacy} is an equivalence class of homomorphisms under this action.  Homomorphisms-modulo-conjugacy can be composed (by composing representatives), so there
is a category $\CGrp$ of groups and homomorphisms-modulo-conjugacy.  Given a pair of finite sets $I\subset J\subset \N$, let $i_I^J\colon I \into J$ denote the inclusion.  

\begin{definition}
\label{def:weakFIgroup}
A \emph{weak FI-group} $G$ consists of the following data.  
\begin{compactenum}[(i)]
\item For each finite set $I \subset \N$, a group $G_I$.
\item For each injection $f\colon I \into J$ between finite sets $I,J \subset \N$,
a homomorphism-modulo-conjugacy $G_f\colon G_I \rightarrow G_J$.  These homomorphisms-modulo-conjugacy
must satisfy the following compatibility conditions.
\begin{compactenum}[a.]
\item For all finite sets $I \subset \N$, we have $G_{\id_I} = \id$.
\item For all finite sets $I,J,K \subset \N$ and all injections $f\colon I \into J$ and
$g\colon J \into K$, we have $G_{g \circ f}$ equal to $G_g \circ G_f$ in $\CGrp$.
\end{compactenum}
\item For each pair of finite sets $I\subset J \subset \N$, a homomorphism
$G_I^J \colon G_I \rightarrow G_J$. These homomorphisms must satisfy the following compatibility conditions.
\begin{compactenum}[a.]
\item For all pairs of finite sets $I\subset J\subset \N$, the homomorphism-modulo-conjugacy $G_{i_I^J}$ is represented by the homomorphism $G_I^J$.
%\item For all finite sets $I \subset \N$, we have $G_I^I = \text{id}$.
\item For all triples of finite sets $I \subset J \subset K\subset \N$, we
have $G_J^K \circ G_I^J = G_I^K$.
\end{compactenum}
\end{compactenum}
\end{definition}
In \S\ref{section:mod}, we will see that the mapping class groups of surfaces with one boundary component
can be naturally viewed as a weak FI-group.

\begin{remark}
\label{remark:forget}
Every FI-group $G$ can be canonically considered as a weak FI-group, by considering the homomorphisms $G_f\colon G_I\to G_J$ only as homomorphisms-modulo-conjugacy (and setting $G_I^J\coloneq G_{i_I^J}$). The conditions of Definition~\ref{def:FIgroup} imply that all the conditions of Definition~\ref{def:weakFIgroup} are satisfied. %Denoting this weak FI-group by $G'$ for the moment, this is done by setting $(G')_I \coloneq G_I$, $(G')_I^J\coloneq G_{i_I^J}$, and taking $(G')_f$ to be the homomorphism-modulo-conjugacy represented by the homomorphism $G_f$.
Throughout the paper, wherever necessary we consider FI-groups as weak FI-groups via this ``forgetful'' process. As a result, many of our technical results and definitions will be stated for weak FI-groups, but they apply equally well to FI-groups.
\end{remark}

\begin{remark}
The conditions on the $G_f$ in Definition~\ref{def:weakFIgroup}(ii) are equivalent to the assertion that they piece together to give a functor
from $\FI$ to $\CGrp$.  In \cite{ChurchEllenbergFarbFI}, such functors were called ``FI-groups up to conjugacy''.
Similarly, the conditions on the homomorphisms $G_I^J$ in Definition~\ref{def:weakFIgroup}(iii) are equivalent to the assertion
that they piece together to give a functor from the category of finite subsets of $\N$ and \emph{inclusions} to $\Grp$.
\end{remark}

\para{FI-modules}
An \emph{FI-module} is an FI-group $W$ such that $W_I$ is an abelian group for all finite sets $I \subset \N$.
We say that an FI-module $W$ has \emph{finite rank} if for all finite sets $I \subset \N$, the 
abelian group $W_I$ is finitely generated.

\begin{remark}
FI-modules were originally defined by the first author  with Ellenberg and Farb in \cite{ChurchEllenbergFarbFI},
and we refer the reader to \cite{ChurchEllenbergFarbFI} for many examples of them.  The paper
\cite{ChurchEllenbergFarbFI} considers FI-modules over an arbitrary ring $R$; in that language,
our FI-modules are FI-modules over the ring $\Z$. Observe that there would be no point in defining ``weak FI-modules'';
indeed, since homomorphisms-modulo-conjugacy  coincide with homomorphisms when the groups involved are abelian, the conditions of Definition~\ref{def:weakFIgroup} reduce to the conditions of Definition~\ref{def:FIgroup} in this case.
\end{remark}

\para{Normal weak FI-subgroups}
Let $A$ and $B$ be groups, and let $f\colon A\to B$ be some homomorphism-modulo-conjugacy. Observe that if $N\normal A$ is a normal subgroup, the subgroup $f(N)\subset B$ is well-defined, even though $f$ is not a well-defined homomorphism. 

\begin{definition}
\label{def:weaksubgroup}
Let $G$ be either an FI-group or a weak FI-group.  A \emph{normal weak FI-subgroup} $H$ of $G$, denoted
$H \normal G$, consists of a normal subgroup $H_I \normal G_I$ for each finite set $I \subset \N$ satisfying
the following property. 
\begin{compactitem}
\item For all injections $f\colon I \into J$ between finite sets $I,J \subset \N$, 
we have $G_f(H_I)\subset H_J$.  
\end{compactitem}
Given $H\normal G$ and $H' \normal G$, we write $H\subset H'$ if $H_I \subset H'_I$ for all finite sets $I \subset \N$.
\end{definition}
\begin{remark}
By the remark preceding Definition~\ref{def:weaksubgroup}, the fact that $H_I\normal G_I$ guarantees that the condition $G_f(H_I)\subset H_J$ is well-defined, even when $G$ is only a weak FI-group. This issue is the reason
we do not define non-normal weak FI-subgroups.
\end{remark}

\begin{remark}
\label{remark:warning}
If $G$ is an FI-group and $H\normal G$, then $H$ is itself an FI-group.  However,
we warn the reader that if $G$ is only a weak FI-group and $H\normal G$, then $H$ is
\emph{not} necessarily a weak FI-group. The reason is that a homomorphism-modulo-conjugacy $G_I\to G_J$ cannot be restricted to a homomorphism-modulo-conjugacy $H_I\to H_J$, since homomorphisms conjugate by an element of $G_J$ need not be conjugate by an element of its subgroup $H_J$.
\end{remark}

\para{Bounded generation} The notion of bounded generation, which we define in this subsection, captures the idea that an FI-group (or weak FI-group) is generated by elements ``supported on subsets of bounded size''.

\begin{definition}
\label{def:FIgroupsupport}
Let $G$ be a weak FI-group, and $H\normal G$. Given a pair of finite sets $I\subset J\subset \N$, we denote by $H_J(I)$ the image $H_{J}(I)\coloneq G_I^J(H_I)\subset H_J$. 
\end{definition}
One should regard $H_{J}(I)$ as the subgroup of $H_J$ which is ``supported on the subset $I$''. Given $I\subset K\subset J$, the identity $G_I^J=G_K^J\circ G_I^K$ implies that $H_J(I)\subset H_J(K)$. Taking $H=G$, we have $G_J(I)\coloneq G_I^J(G_I)\subset G_J$.

\begin{definition}
\label{def:boundedgen}
Let $G$ be a weak FI-group. Given $A\geq 0$, we say that $G$ is \emph{boundedly generated in degree $A$} if for all finite sets $J\subset \N$,
\begin{equation}
\label{eq:boundedgencondition}
G_J\text{ is generated by its subgroups }G_J(I)\text{ for those }I\subset J\text{ satisfying }\abs{I}\leq A.
\end{equation}
%the group $G_J$ is generated by its subgroups $G_J(I)$ for those $I\subset J$ satisfying $\abs{I}\leq A$.
We say that $G$ is \emph{boundedly generated} if \eqref{eq:boundedgencondition} holds for some $A\geq 0$.
\end{definition}

\begin{lemma}
\label{lemma:boundedgen}
Let $G$ be an FI-group.  Fix $A\geq 0$, and assume that for all $n\in \N$, the condition \eqref{eq:boundedgencondition} holds
for the set $J=[n]$.  Then $G$ is boundedly generated in degree $A$.
\end{lemma}
\begin{proof}
Given any set $J\subset \N$, let $n=\abs{J}$, and choose a bijection $f\colon [n] \rightarrow J$. 
Given any $I\subset J$, set $I'\coloneq f^{-1}(I)\subset [n]$. By Definition~\ref{def:FIgroupsupport} we have \[G_f(G_{[n]}(I))=G_f(G_{I'}^{[n]}(G_{I'}))=G_I^J(G_f(G_{I'}))=G_I^J(G_I)=G_J(I),\] where the equality $G_f\circ G_{I'}^{[n]}=G_I^J\circ G_f$ holds because $G$ is an FI-group. Therefore the condition \eqref{eq:boundedgencondition} for $J$ follows from condition \eqref{eq:boundedgencondition} for $[n]$.
\end{proof}

\begin{remark}
\label{remark:weakboundedgen}
When $G$ is a weak FI-group it is \emph{not enough} to check \eqref{eq:boundedgencondition} for $J=[n]$. The proof of Lemma~\ref{lemma:boundedgen} breaks down not just because $G_f\circ G_{I'}^{[n]}=G_I^J\circ G_f$ need not hold, but because $G_f(G_{[n]}(I))$ is not even a well-defined subgroup. The best we could conclude is that $G_J$ is \emph{normally} generated by the subgroups $G_J(I)$ with $\abs{I}\leq A$, a far weaker condition. Indeed, choosing the homomorphisms $G_I^J$ so that a given weak FI-group is boundedly generated can be quite delicate. This issue is the main reason that we must be so careful in \S\ref{section:mod} when making the Torelli group into a weak FI-group.
\end{remark}

\begin{definition}
\label{def:boundednormalgen}
Let $G$ be a weak FI-group, and let $H\normal G$ be a normal weak FI-subgroup. Given $B\geq 0$, we say that $H\normal G$ is \emph{boundedly normally generated in degree $B$} if for all finite sets $J\subset \N$,
\begin{equation}
\label{eq:boundednormalgencondition}
\begin{split}
    H_J\text{ is generated by the $G_J$-conjugates of its subgroups }H_J(I)\qquad\qquad \\
    \text{ for those }I\subset J\text{ satisfying }\abs{I}\leq B.
  \end{split}
\end{equation}
%the group $H_J$ is generated by the $G_J$-conjugates of the subgroups $H_J(I)$ for those $I\subset J$ satisfying $\abs{I}\leq B$.
We say that $H\normal G$ is \emph{boundedly normally generated} if this holds for some $B\geq 0$.
\end{definition}

\begin{remark}
The condition \eqref{eq:boundednormalgencondition} is vacuous for $\abs{J}\leq B$, since $H_J=H_J(J)$; similarly the condition \eqref{eq:boundedgencondition} is vacuous for $\abs{J}\leq A$.
\end{remark}

\para{Central filtrations}
Let $G$ be a weak FI-group. Given $H \normal G$, we can define $[G,H]$ via the formula $[G,H]_I = [G_I,H_I]$
for finite sets $I \subset \N$; it is easy to check that $[G,H]\normal G$ and $[G,H]\subset H$.

\begin{definition}
\label{definition:centralfiltration}
Let $G$ be a weak FI-group.  A \emph{central filtration} of $G$ consists of
normal weak FI-subgroups $G(k) \normal G$ for each $k \geq 1$ satisfying
\[G=G(1)\supset G(2)\supset \cdots\supset G(k)\supset G(k+1)\supset \cdots\]
and $[G,G(k)]\subset G(k+1)$ for all $k\geq 1$.  This latter condition implies that
$G(k)_I / G(k+1)_I$ is an abelian group for all finite sets $I \subset \N$, and we say
that our central filtration is of \emph{finite rank} if the abelian group $G(k)_I/G(k+1)_I$ is finitely generated
for all $k \geq 1$ and all finite sets $I \subset \N$.
\end{definition}

We can now state our main technical theorem, which we will prove in \S\ref{section:generalproof} below.
\begin{boundedgentheorem}
Let $G$ be a weak FI-group with a central filtration $\{G(k)\}_{k=1}^\infty$ of finite rank. If $G$ is boundedly generated, then $G(k)\normal G$ is boundedly normally generated for every $k\geq 1$.
\end{boundedgentheorem}

\subsection{Technical results about FI-groups}
\label{section:technical}

This section collects a number of technical results about FI-groups that we will need in the proof of Theorem~\ref{maintheorem:boundedgeneration}.

\para{Controlling the support}
We begin with the following lemma, which allows us to control the support of certain commutators. 
\begin{lemma}
\label{lemma:union}
Let $G$ be a weak FI-group with a central filtration $\{G(k)\}_{k=1}^{\infty}$.  Fix some $k \geq 1$ and
let $I,{I'},J \subset \N$ be finite sets satisfying $I,{I'} \subset J$.  Consider $w \in G_J(I)$ and $z \in G(k)_J(I')$.
Then $[w,z] \in G(k+1)_J(I \cup I')$.
\end{lemma}
Lemma~\ref{lemma:union} follows immediately from the inclusions $G_J(I)\subset G_J(I\cup I')$ and $G(k)_J(I')\subset G(k)_J(I\cup I')$ together with the definition of a central filtration.

\para{Subgroups normally generated on sets of a fixed size}
Let $G$ be a weak FI-group with $H \normal G$, and fix $N \geq 0$.  For
each finite set $J \subset \N$, define $H^{\leq N}_J$ to be the subgroup generated by the $G_J$-conjugates
of the subgroups $H_J(I)$ for those $I \subset J$ satisfying $\abs{I} \leq N$.  Since $H_J \normal G_J$ is a normal subgroup, we have $H^{\leq N}_J \subset H_J$.

\begin{lemma}
\label{lemma:HleqN}
Let $G$ be a weak FI-group, let $H \normal G$ be a normal weak FI-subgroup, and fix $N \geq 0$.  Then
$H^{\leq N}\normal G$ is a normal weak FI-subgroup of $G$.
\end{lemma}
Comparing the definition of $H^{\leq N}$ with \eqref{eq:boundednormalgencondition}, we see that by definition
\begin{equation}
\label{eq:boundednormalgenequiv}
H\normal G\text{ is boundedly normally generated in degree $N$}\qquad\iff\qquad H=H^{\leq N}.
\end{equation}
\begin{proof}[Proof of Lemma~\ref{lemma:HleqN}]
We must prove that for any injection $f\colon J \into K$ between finite sets $J,K \subset \N$, we have $G_f(H^{\leq N}_J) \subset H^{\leq N}_K$. Choose a homomorphism representing the homomorphism-modulo-conjugacy $G_f$, which by abuse of notation we also denote $G_f$. Since $H^{\leq N}_K$ is a normal subgroup of $G_K$, it is enough to show that
$G_f(H_J(I)) \subset H^{\leq N}_K$ for all $I \subset J$ with $\abs{I} \leq N$.  

%Fix such a subset $I\subset J$ with $\abs{I}\leq N$, and 
Set $I'\coloneq f(I)\subset K$ and $f'\coloneq f|_I\colon I\to I'$, and choose a representative homomorphism $G_{f'}\colon G_I\to G_{I'}$. Since $f'\colon I\to I'$ is invertible, $G_{f'}$ must be an isomorphism, and restricts to an isomorphism $H_I\xrightarrow{\cong}H_{I'}$.

Definition~\ref{def:weakFIgroup}(ii) implies that $G_f\circ G_I^J$ is $G_K$-conjugate to $G_{I'}^K\circ G_{f'}$. Therefore
\begin{equation}
\label{eq:HleqN}
G_f(H_J(I))=G_f(G_I^J(H_I))\text{\ \ is $G_K$-conjugate to\ \ }G_{I'}^K(G_{f'}(H_I))=G_{I'}^K(H_{I'})=H_K(I').
\end{equation}
Since $\abs{I'}=\abs{I}\leq N$, certainly $H_K(I')$ is contained in $H_K^{\leq N}$ (being among its normal generators). Since $H_K^{\leq N}$ is normal in $G_K$, any $G_K$-conjugate of this subgroup is also contained in $H_K^{\leq N}$. We conclude that $G_f(H_J(I))\subset H_K^{\leq N}$, as desired.
\end{proof}

\para{The  graded quotients of a central filtration}
In this paper, the key examples of FI-modules are the  graded quotients of a central filtration
of a weak FI-group.  The following lemma asserts that these do indeed form FI-modules.

\begin{lemma}
\label{lemma:quotientficentral}
Let $G$ be a weak FI-group and let $\{G(k)\}_{k=1}^{\infty}$ be a central filtration 
of $G$.  Fix some $k \geq 1$.  For each finite set $I \subset \N$, define 
$Q(k)_I \coloneq G(k)_I/G(k+1)_I$.  Then the weak FI-group structure on $G$ induces an FI-module structure
on $Q(k)$.
\end{lemma}

Lemma~\ref{lemma:quotientficentral} is a special case of the following more general lemma.

\begin{lemma}
\label{lemma:quotientficentralgen}
Let $G$ be a weak FI-group and assume that $K\normal G$ and $H \normal G$ satisfy
$[G,H]\subset K \subset H$. Then there exists an FI-module $Q$ defined as follows: for each finite set $I \subset \N$
define $Q_I \coloneq  H_I/K_I$, and for each injection $f\colon I\into J$ let $Q_f\colon Q_I\to Q_J$ be the map induced by $G_f\colon H_I\to H_J$.
\end{lemma}
\begin{proof}
For all finite sets $I \subset \N$, we have $[H_I,H_I] \subset [G_I,H_I] \subset K_I$, so $Q_I$
is an abelian group. It remains to prove that the maps $Q_f$ are well-defined, and that they satisfy the conditions of Definition~\ref{def:FIgroup}(ii).

Consider an injection $f\colon I \into J$ between finite
sets $I,J \subset \N$. The key to the lemma is that, since $[G_J,H_J] \subset K_J$, the conjugation
action of $G_J$ on $H_J$ descends to the trivial action on $Q_J$. Therefore even though $G_f\colon H_I\to H_J$ is only defined up to $G_J$-conjugacy, it descends to a well-defined homomorphism $Q_f\colon Q_I\to Q_J$. Given another injection $g\colon J\into K$, Definition~\ref{def:weakFIgroup}(ii) guarantees that $G_g\circ G_f$ is $G_K$-conjugate to $G_{g\circ f}\colon H_I\to H_K$. It follows that the induced maps $Q_g\circ Q_f$ and $Q_{g\circ f}\colon Q_I\to Q_K$ coincide, so $Q$ is an FI-group.
\end{proof}

\subsection{Central stability and FI-modules}
\label{section:centralstability}
To prove Theorem~\ref{maintheorem:boundedgeneration}, we will need the notion of \emph{central stability},
which was introduced by the second author in \cite{PutmanRepStabilityCongruence}.
The definitions in \cite{PutmanRepStabilityCongruence}
were in terms of the representation theory of the symmetric group.  Here we give an equivalent
definition in the language of FI-modules.  

\para{Bounded generation}
Let $W$ be an FI-module, so all the groups $W_I$ are abelian.
In this case, for any finite set $J \subset \N$
we have a map
\begin{equation}
\label{eq:FIgen}
\bigoplus_{\substack{I \subset J,\\ \abs{I} \leq A}} W_I \longrightarrow W_J
\end{equation}
induced by the homomorphisms $W_I^J\colon W_I \rightarrow W_J$. Definition~\ref{def:boundedgen} says that $W$ is boundedly generated in degree $A$ if \eqref{eq:FIgen} is surjective for every finite set $J\subset \N$. 
(In \cite[Definition 2.14]{ChurchEllenbergFarbFI}, the term ``generated in degree $\leq A$'' was used instead.)

\para{Central stabilization}
Let $W$ be an FI-module, and consider some finite set $J \subset \N$.
We have a homomorphism
\[\psi\colon \bigoplus_{\substack{I \subset J,\\ \abs{I} = \abs{J}-1}} W_I \longrightarrow W_J.\]
If $\abs{J} > A$, then the map \eqref{eq:FIgen} factors through $\psi$,
so $\psi$ is surjective if $W$ is boundedly generated in some degree less than $\abs{J}$.
We wish to understand the kernel of $\psi$.  One source
of elements in $\ker(\psi)$ is as follows.  Consider a finite set $K \subset \N$ such that $K \subset J$
and $\abs{K} = \abs{J}-2$.  Let $I_1,I_2 \subset \N$ be the two distinct sets satisfying $K \subset I_i \subset J$
and $\abs{I_i} = \abs{J}-1$.  We then have a commutative diagram
\[\xymatrix{
                                                    & W_{I_1} \ar[rd]^{W_{I_1}^J} &     \\
W_K \ar[ru]^{W_K^{I_1}} \ar[rd]_{W_K^{I_2}} &                                 & W_J \\
                                                    & W_{I_2} \ar[ru]_{W_{I_2}^J} &     }\]
There is thus a map $W_K \rightarrow \ker(\psi)$ that takes $x \in W_K$ to 
\[\big(W_K^{I_1}(x),\,-W_K^{I_2}(x)\big)\in W_{I_1}\oplus W_{I_2}\subset \bigoplus_{\abs{I} = \abs{J}-1} W_I. \] Collecting all of these maps, we obtain
a map
\[\eta\colon \bigoplus_{\substack{K \subset J,\\ \abs{K} = \abs{J}-2}} W_K \longrightarrow \bigoplus_{\substack{I \subset J,\\ \abs{I}=\abs{J}-1}} W_I\]
whose image lies in $\ker(\psi)$.
The \emph{$J$-central stabilization} of $W$, denoted $\Stab{W}{J}$, is the cokernel of $\eta$.

There is a natural homomorphism $\Stab{W}{J} \rightarrow W_J$, which is surjective if $W$ is boundedly generated
in some degree less than $\abs{J}$. A morphism $\Psi\colon V\to W$ of FI-modules induces a map $\Stab{V}{J}\to \Stab{W}{J}$ consistent with the map $\Psi_J\colon V_J\to W_J$ and the maps $\Stab{V}{J}\to V_J$ and $\Stab{W}{J}\to W_J$.

\para{Central stability}
We say that an FI-module $W$ is \emph{centrally stable} starting at $E \geq 0$ if for all finite 
sets $J \subset \N$
with $\abs{J} > E$, the natural map $\Stab{W}{J} \rightarrow W_J$ is an isomorphism.  This implies
in particular that $W$ is boundedly generated in degree $E$.  We say that $W$ is
\emph{centrally stable} if it is centrally stable starting at some $E$.  One should
think of a centrally stable FI-module as being ``finitely presented''.
The key technical result underpinning this paper
is the following theorem of the first author with Ellenberg, Farb, and Nagpal.  It should
be viewed as a ``Noetherian'' property of FI-modules.

\begin{proposition}[{\cite[Corollary 2.11]{ChurchEllenbergFarbNagpal}}]
\label{proposition:noetheriancentral}
Let $W$ be a finite-rank FI-module.  If $W$ is boundedly generated, then $W$ is centrally stable.
\end{proposition}

\para{The power of central stability}
If $W$ is an FI-module which is centrally stable starting at $E$, then $W$ is determined by its initial
segment of size $E$, by which we mean the groups $W_J$ for finite sets $J \subset \N$ with $\abs{J} \leq E$ 
and the maps between
these groups.  One way of using this is as follows.

\begin{lemma}
\label{lemma:isomorphism}
Let $\Psi\colon V \rightarrow W$ be a morphism between FI-modules.  Assume that $W$ is centrally stable
starting at $E \geq 0$, that $V$ is boundedly generated in degree $E$, and that for all finite sets $J \subset \N$ with $\abs{J} \leq E$, 
the map $\Psi_J\colon V_J \rightarrow W_J$ is an isomorphism.  Then $\Psi$ is an isomorphism.
\end{lemma}
\begin{proof}
We will prove that $\Psi_J\colon V_J \rightarrow W_J$ is an isomorphism for all finite sets $J \subset \N$
by induction on $\abs{J}$.  The base cases are when $\abs{J} \leq E$, where $\Psi_J$ is an isomorphism by assumption. Assume now that $\abs{J} > E$ and that $\Psi_I$ is an isomorphism for all sets $I$ with $\abs{I}<\abs{J}$. Consider the commutative diagram:
\[\xymatrix{
\Stab{V}{J} \ar@{->>}[r] \ar_{\cong}[d] &V_J \ar^{\Psi_J}[d]\\
 \Stab{W}{J}\ar[r]^(0.58){\cong} &W_J
}
\]
The first vertical map $\Stab{V}{J}\to \Stab{W}{J}$ is 
an isomorphism because $\Psi_I$ is an isomorphism whenever $\abs{I}<\abs{J}$. The first horizontal map is surjective because $V$ is boundedly generated in degree $E<\abs{J}$,
and the second horizontal map is an isomorphism because $W$ is centrally
stable starting at $E<\abs{J}$. We conclude that $\Psi_J$ is an isomorphism, as desired.
\end{proof}

\subsection{Proof of Theorem~\ref{maintheorem:boundedgeneration}}
\label{section:generalproof}

In this section, we prove Theorem~\ref{maintheorem:boundedgeneration}.

Let $G$ be a weak FI-group with a central filtration $\{G(k)\}_{k=1}^\infty$ of finite rank. Assume that $G$ is boundedly generated in degree $A$. Our goal is to prove  for each $k\geq 1$ that $G(k)\normal G$ is boundedly normally generated. Via the equivalence \eqref{eq:boundednormalgenequiv}, we must prove that for each $k\geq 1$ there exists some $B_k\geq 0$ such that $G(k)^{\leq B_k}=G(k)$.

We will prove this by induction on $k$. In the base case $k=1$ we have  $G=G(1)$, so we may take $B_1\coloneq A$.
Now assume that for some fixed $k \geq 1$, we have constructed some $B_k \geq 0$ such that $G(k)^{\leq B_k}=G(k)$. 
We will find $B_{k+1} \geq 0$ such that $G(k+1)^{\leq B_{k+1}}=G(k+1)$; this will complete the inductive step.

%As was discussed right before Definition~\ref{definition:centralfiltration}, we have a normal weak FI-subgroup $[G,G(k)]$ of $G$. 
 Since the $G(k)$ form a central filtration of $G$, we know that 
$[G,G(k)]\subset G(k+1)$.  Our first step will be to improve this inclusion.

\BeginClaims
\begin{claims}
For all $N \geq A+B_k$, we have $[G,G(k)] \subset G(k+1)^{\leq N}$.
\end{claims}
\begin{proof}[Proof of claim] We will use the notation $a^b=b^{-1}ab$ and $[a,b]=a^{-1}b^{-1}ab=a^{-1}a^b$.
Fix some $N \geq A+B_k$, and consider a finite set $J \subset \N$.  
By definition, $[G,G(k)]_J$ is generated by the
set
\begin{equation}
\label{eqn:gen1}
\Set{$[x,y]$}{$x \in G_J$, $y \in G(k)_J$}.
\end{equation}
Our inductive hypothesis says that $G(k)^{\leq B_k}_J = G(k)_J$, so we can write
$y \in G(k)_J$ as a product of elements of the set
\[\Set{$z^g$}{$g \in G_J$, $z \in G(k)_J(I)$ for $I \subset J$ with $\abs{I} \leq B_k$}.\]
Repeatedly applying the Witt--Hall commutator identity $[a,bc] = [a,c] \cdot [a,b]^c$, we can therefore
express every element of \eqref{eqn:gen1} as a product of $G_J$-conjugates of elements of the set
\begin{equation}
\label{eqn:gen2}
\Set{$[x,z^g]$}{$x,g \in G_J$, $z \in G(k)_J(I)$ for $I \subset J$ with $\abs{I} \leq B_k$}.
\end{equation}
Consider some $[x,z^g]$ as in \eqref{eqn:gen2}.  We have $[x,z^g] = [x^{g^{-1}},z]^g$.  Since $G$ is boundedly generated in degree $A$, we can write $x^{g^{-1}} \in G_J$ as a product of elements in the set
\[\Set{$w$}{$w \in G_J({I'})$ for some ${I'} \subset J$ with $\abs{I'} \leq A$}.\]
Repeatedly applying the Witt--Hall commutator identity $[ab,c] = [a,c]^b \cdot [b,c]$, we can
therefore express $[x,z^g]$ as a product of $G_J$-conjugates of elements
of the set
\begin{equation}
\label{eqn:gen3}
\Set{$[w,z]$}{$w \in G_J({I'})$ for ${I'} \subset J$ with $\abs{I'} \leq A$, $z \in G(k)_J(I)$ for $I \subset J$ with $\abs{I} \leq B_k$}.
\end{equation}
In summary, $[G,G(k)]_J$ is generated by the $G_J$-conjugates of elements in \eqref{eqn:gen3}. By Lemma 
\ref{lemma:union}, every element in \eqref{eqn:gen3} lies in $G(k+1)^{\leq N}$, so this concludes the proof of Claim~1.\end{proof}

Lemma~\ref{lemma:quotientficentral} yields an FI-module $W(k) \coloneq  G(k)/G(k+1)$;  the assumption that the central filtration $\{G(k)\}_{k=1}^\infty$ is of finite rank says precisely that the FI-module $W(k)$ is of finite rank. 
Also, combining
Claim~1 with Lemma~\ref{lemma:quotientficentralgen}, we obtain for any $N \geq A+B_k$ an FI-module
$V^N(k) \coloneq G(k) / G(k+1)^{\leq N}$. We warn the reader that we do
\emph{not} yet know that $V^N(k)$ is of finite rank.

\begin{claims}
For  $N \geq A + B_k$, both $V^N(k)$ and $W(k)$ are boundedly generated in degree $B_k$.
\end{claims}
\begin{proof}[Proof of claim]
Fix  $N \geq A+B_k$.  Since $W(k)$ is a quotient of $V^N(k)$, it suffices to prove that $V^N(k)$ is boundedly generated in degree $B_k$.  Consider a finite set $J \subset \N$.
There is a surjective map $\rho\colon G(k)_J \rightarrow V^N(k)_J$.  Given 
$x \in G(k)_J$ and $y \in G_J$, Claim~1 implies that $[x,y]\in G(k+1)^{\leq N}=\ker(\rho)$, so $\rho(x) = \rho(y^{-1} x y)$.
Our inductive hypothesis says that $G(k)_J = G(k)^{\leq B_k}_J$, i.e.\ that $G(k)_J$ is generated by the $G_J$-conjugates of $G(k)_J(I)$ for $\abs{I}\leq B_k$. We conclude that
$V^N(k)_J=\rho(G(k)_J)$ is generated by
\[\Set{$\rho(G(k)_J(I))$}{$I \subset J$, $\abs{I} \leq B_k$} = \Set{$V^N(k)_J(I)$}{$I \subset J$,  $\abs{I} \leq B_k$},\]
as desired. This concludes the proof of Claim~2.
\end{proof}

The FI-module $W(k)$ is finite rank by assumption, and it is boundedly generated by Claim~2, so  
Proposition~\ref{proposition:noetheriancentral} implies that $W(k)$ is centrally stable. Choose $B_{k+1}$ (which we
may take to be at least $A+B_k$) such that
$W(k)$ is centrally stable starting at $B_{k+1}$.

We have an FI-module morphism $\pi\colon V^{B_{k+1}}(k)\onto W(k)$, since $G(k+1)^{\leq B_{k+1}}_J\subset G(k+1)_J$ for any finite set $J\subset \N$. Note that the kernel of $\pi_J\colon V^{B_{k+1}}(k)_J \rightarrow W(k)_J$ is isomorphic
to $G(k+1)_J / G(k+1)^{\leq B_{k+1}}_J$. 

If $\abs{J}\leq B_{k+1}$, by definition $G(k+1)^{\leq B_{k+1}}_J= G(k+1)_J$, so in this case $\pi_J\colon V^{B_{k+1}}(k)_J\to W(k)_J$ is an isomorphism. Moreover $V^{B_{k+1}}(k)$ is boundedly generated in degree $B_k\leq B_{k+1}$ by Claim~2. Applying Lemma~\ref{lemma:isomorphism}, we conclude that 
$\pi^{B_{k+1}}\colon V^{B_{k+1}}(k)\onto W(k)$ is an isomorphism. 

We conclude that  $\ker(\pi_J)\cong G(k+1)_J / G(k+1)^{\leq B_{k+1}}_J$ is trivial for all finite sets $J\subset \N$. In other words, we have $G(k+1)= G(k+1)^{\leq B_{k+1}}$; by \eqref{eq:boundednormalgenequiv}, this means that $G(k+1)$ is boundedly normally generated in degree $B_{k+1}$. This finishes the proof of the inductive step, and thus concludes the proof of Theorem~\ref{maintheorem:boundedgeneration}.

\begin{remark}
Theorem~\ref{maintheorem:boundedgeneration} gives no bound whatsoever on the constants $B_k$, and it is not possible to obtain any such bounds from our proof. The reason is in our use of Proposition~\ref{proposition:noetheriancentral}, which rests on the Noetherian property of FI-modules proved in \cite[Corollary 2.11]{ChurchEllenbergFarbNagpal}. This property is non-constructive, since it ultimately relies on the Noetherian property of the ring $\Z$. As a result we have no way to know how large the constant $B_{k+1}$ must be taken in the inductive step.
\end{remark}

\section{Automorphism groups of free groups}
\label{section:aut}
We begin in \S\ref{section:autfi} by
showing how to assemble all the different automorphism groups of free groups into an FI-group. In \S\ref{section:autgen} we discuss generators for $\IA_n$ and prove Theorem~\ref{maintheorem:autgentorelli}.

\subsection{Automorphism groups of free groups as an FI-group}
\label{section:autfi}

In this section, we show how the automorphism groups of free groups fit together into
an FI-group.  We also show that a similar result holds for their Torelli subgroups and that
the Johnson filtration gives a central filtration of this FI-group.

\para{Automorphism groups of free groups}
We first define an FI-group $\AF$ which collects together the automorphism groups
of free groups of different ranks as follows.
\begin{compactitem}
\item For each finite set $I \subset \N$, let $F_I$ be the free group on the
set $\Set{$x_i$}{$i \in I$}$ and define $\AF_I = \Aut(F_I)$.
\item For each injection $f\colon I \into J$ between finite sets $I,J \subset \N$,
define an injection $\psi_f\colon F_I \into F_J$ via the formula $\psi_f(x_i) = x_{f(i)}$ for $i \in I$.
We then define the homomorphism $\AF_f\colon \AF_I \rightarrow \AF_J$ via the formula
\begin{equation}
\label{eqn:AFf}
\AF_f(\varphi)(x_j) = \begin{cases}
\psi_f \circ \varphi \circ \psi_f^{-1}(x_j) & \text{if $j \in f(I)$,}\\
x_j & \text{if $j \notin f(I)$.}
\end{cases}
\end{equation}
\end{compactitem}
It is clear that these homomorphisms $\AF_f$ satisfy the compatibility condition in Definition~\ref{def:FIgroup}(ii), so this defines an FI-group $\AF$.

\para{The Johnson filtrations}
For $k \geq 1$, we define $\IA(k)\normal \AF$ as follows.
For each finite set $I \subset \N$, define $\IA(k)_I\normal \AF_I$ to
be the kernel of the action of $\AF_I=\Aut(F_I)$ on $F_I / \gamma_{k+1}(F_I)$.
The following lemma implies that $\IA(k)\normal \AF$.

\begin{lemma}
If $I,J \subset \N$ are finite sets and $f\colon I \into J$ is an injection,
then $\AF_f(\IA(k)_I) \subset \IA(k)_J$.
\end{lemma}
\begin{proof}
We have a natural splitting
$F_J = F_{f(I)} \ast F_{J-f(I)}$.  Consider $\varphi \in \IA(k)_I$.
Since the injection $\psi_f\colon F_I \rightarrow F_J$ takes $\gamma_{k+1}(F_I)$
into $\gamma_{k+1}(F_{f(I)}) \subset \gamma_{k+1}(F_J)$, the automorphism $\AF_f(\varphi)$ acts as the identity
on the image of $F_{f(I)}$ in $F_J / \gamma_{k+1}(F_J)$.  The automorphism
$\AF_f(\varphi)$ also acts as the identity on $F_{J-f(I)}$, and thus certainly acts
as the identity on its image in $F_J/\gamma_{k+1}(F_J)$. Since the images of $F_{f(I)}$ and $F_{J-f(I)}$ generate $F_J/\gamma_{k+1}(F_J)$, we conclude that $\AF_f(\varphi) \in \IA(k)_J$.
\end{proof}

Since $\AF$ is an FI-group (and not merely a weak FI-group), $\IA(k)$ is itself
an FI-group.  Note that for the set $[n]\subset \N$ we have $F_{[n]}=F_n$, so $\AF_{[n]}=\Aut(F_n)$, $\IA(1)_{[n]}=\IA_n$, and $\IA(k)_{[n]}=\IA_n(k)$.

\begin{proposition}
\label{prop:IAk}
$\{\IA(k)\}_{k=1}^\infty$ is a central filtration of
$\IA(1)$ of finite rank.
\end{proposition}
\begin{proof}
Fix a finite set $I \subset \N$.  Since $\gamma_2(F_I)\supset \gamma_3(F_I)\supset \gamma_4(F_I)\supset \cdots$, we clearly have
\[\IA(1) \supset \IA(2) \supset \IA(3) \supset \cdots.\]
For $k \geq 1$, we must show that
$[\IA(1)_I,\IA(k)_I] \subset \IA(k+1)_I$ and
that $\IA(k)_I/\IA(k+1)_I$ is a finite-rank abelian group.
Setting $n = \abs{I}$, the evident isomorphism $\IA(1)_I \cong \IA_n(1)$ takes
$\IA(k)_I$ to $\IA_n(k)$ for all $k \geq 1$.  Our claim is thus
equivalent to showing for all $k \geq 1$ that
$[\IA_n(1),\IA_n(k)] \subset \IA_n(k+1)$ and that
$\IA_n(k)/\IA_n(k+1)$ is a finite-rank abelian group.

For this, we will need the higher Johnson homomorphisms.
For all $k \geq 1$, let $\Lie_k(\Z^n)$ denote the $k^{\text{th}}$ graded piece of the free
Lie algebra on $\Z^n$.  The $k^{\text{th}}$ Johnson homomorphism is then a homomorphism
$\tau_k\colon \IA_n(k)\rightarrow \Hom(\Z^n,\Lie_{k+1}(\Z^n))$.  We will say more
about $\tau_k$ in \S \ref{section:bounds}; right now, we only need the following
properties (see Satoh~\cite{SatohSurvey} for a survey).
\begin{compactenum}[(I)]
\item The kernel of $\tau_k$ equals $\IA_n(k+1)$.
\item For $\psi \in \IA_n(1) = \IA_n$ and $\varphi \in \IA_n(k)$, we have
$\tau_k(\psi \varphi \psi^{-1}) = \tau_k(\varphi)$.
\end{compactenum}
Property (I) implies that $\IA_n(k) / \IA_n(k+1)$ is a subgroup of
$\Hom(\Z^n,\Lie_{k+1}(\Z^n))$, and in particular is a finite-rank abelian group.
Property (II) implies that $\tau_k([\IA_n(1),\IA_n(k)])=0$, so Property (I)
implies that $[\IA_n(1),\IA_n(k)] \subset \IA_n(k+1)$.
\end{proof}

\subsection{Generating the Torelli subgroup of \texorpdfstring{$\Aut(F_n)$}{Aut(Fn)} and its Johnson filtration}
\label{section:autgen}

As was discussed in the introduction, Magnus~\cite{MagnusGenerators} gave a
finite generating set for $\IA_n$.  We will need a corollary of his result.
Given a splitting $F=A \ast B$, recall that an automorphism $\varphi$ of $F$ is \emph{supported on the splitting} $A \ast B$ if $\varphi(A)=A$ and $\varphi|_B=\id$.  Given a pair of finite sets
$I \subset J \subset \N$, it is clear from \eqref{eqn:AFf} that $\AF_J(I)$ is exactly
the subgroup of $\AF_J = \Aut(F_J)$ consisting of automorphisms that are supported
on the splitting $F_J = F_I \ast F_{J-I}$.  Similarly, $\IA(k)_J(I) = \IA(k)_J \cap \AF_J(I)$ consists
of those automorphisms in $\IA(k)_J$ that are supported on this splitting.

In the case $J=[n]$, we write $\Aut(F_n,I)$ for $\AF_{[n]}(I)$, and write $\IA_n(I)$ for $\IA(1)_{[n]}(I)=\IA_n\cap \Aut(F_n,I)$. For example, recall the automorphisms $c_{ij},m_{ijk}\in \IA_n$ defined in the introduction:
\[c_{ij}(x_\ell) = \begin{cases}
x_j^{-1} x_\ell x_j & \text{if $\ell=i$},\\
x_\ell & \text{otherwise}.
\end{cases}\qquad\qquad m_{ijk}(x_{\ell}) = \begin{cases}
x_{\ell} [x_j,x_k] & \text{if $\ell=i$},\\
x_{\ell} & \text{otherwise}.
\end{cases}\]
Clearly $c_{ij}$
is supported on the splitting $\langle x_i,x_j\rangle \ast\langle x_\ell\,|\,\ell\neq i,j\rangle$, so $c_{ij}\in \IA_n(\{i,j\})$.
Similarly, the automorphism $m_{ijk}\in \IA_n$  is supported on the splitting
$\langle x_i,x_j,x_k \rangle \ast \langle x_\ell\,|\,\ell\neq i,j,k\rangle$, so $m_{ijk}\in \IA_n(\{i,j,k\})$.

Since Magnus proved that the elements $c_{ij}$ and $m_{ijk}$ generate $\IA_n$ for all $n$, we have the following proposition.

\begin{proposition}[Generators for $\IA_n$]
\label{proposition:autgentorelli}
For any $n\geq 0$, the group $\IA_n$ is generated by the subgroups
\[\big\{\!\IA_n(I)\,\,\big|\,\,I\subset \{1,\ldots,n\}\text{ satisfies }\ \abs{I}\leq 3\big\}.\]
\end{proposition}

%\subsection{Proof of Theorem~\ref{maintheorem:autgentorelli}}
%\label{section:autproof}

We are now ready to prove Theorem~\ref{maintheorem:autgentorelli}.

\begin{proof}[{Proof of Theorem~\ref{maintheorem:autgentorelli}}]
$\IA(1)$ is an FI-group, so applying Lemma~\ref{lemma:boundedgen}, Proposition~\ref{proposition:autgentorelli} implies that $\IA(1)$ is boundedly generated in degree $A=3$. Proposition~\ref{prop:IAk} states that $\{\IA(k)\}_{k=1}^\infty$ is a central filtration of bounded rank. Applying Theorem~\ref{maintheorem:boundedgeneration}, we conclude that for all $k\geq 1$, there exists $B_k\geq 0$ so that $\IA(k)\normal \IA(1)$ is boundedly normally generated in degree $B_k$. 

Let us apply this conclusion to $\IA(k)_{[n]}=\IA_n(k)$. The bounded normal generation of $\IA(k)\normal \IA(1)$ states that $\IA(k)_{[n]}$ is generated by the $\IA_n$-conjugates of its subgroups $\IA(k)_{[n]}(I)$ for those $I\subset [n]$ with $\abs{I}\leq B_k$.

We saw above that $\IA(k)_{[n]}(I)$ consists of those automorphisms in $\IA_n(k)$ which are supported on the splitting $F_n=F_I\ast F_{[n]-I}$.
%; in other words, $\IA(k)_{[n]}(I)=\IA_n(k)\cap \Aut(F_n,I)$.
% The $\IA_n$-conjugates of this subgroup 
%are thus of the form $\IA_n(k)\cap (\varphi \Aut(F_n,I)\varphi^{-1})$ for $\varphi\in \IA_n$.  
%The group $\varphi \Aut(F_n,I)\varphi^{-1}$
The $\varphi$-conjugate of this subgroup thus consists of those automorphisms in $\IA_n(k)$ supported on the splitting 
\[F_n=\varphi(F_I)\ast\varphi(F_{[n]-I}).\] 
When $\varphi\in \IA_n$, this is a homologically standard splitting.  Therefore $\IA_n(k)$ 
is generated by elements of $\IA_n(k)$ supported on homologically standard splittings of rank $\leq B_k$, as desired.
\end{proof}

\section{Mapping class groups}
\label{section:mod}

We begin in \S\ref{section:modfi} by
showing how to assemble all the different mapping class groups for surfaces of different genus into a weak FI-group. We also show that we can do the same for their Torelli subgroups and that
the Johnson filtration gives a central filtration of this weak FI-group. In \S\ref{section:modgen} we establish a generating set for $\Torelli_g^1$, prove Theorem~\ref{maintheorem:smallgenset}, and
finally prove Theorem~\ref{maintheorem:modgentorelli}.

\subsection{Mapping class groups as a weak FI-group}
\label{section:modfi}

Ideally, we would like to construct an FI-group $\Mod$ such that $\Mod_{[g]}\cong \Mod_g^1$ 
and such that $\Mod_{\{i\}}$ is the subgroup supported on the ``$i^{\text{th}}$ handle''. 
Unfortunately, this is not possible, for the following reason.

Recall that any FI-group $G$ has an action of the symmetric group $\Sym_n$ 
on the group $G_{[n]}$.  If there did exist an FI-group $\Mod$ as above, then 
the subgroups $\Mod_{\{i\}}$ would be permuted by the action of the symmetric 
group $\Sym_g$ on $\Mod_g^1$.  Since these subgroups are disjoint, this 
action must be faithful. However, this is impossible.  Indeed, for $g \geq 2$, it follows from
work of Ivanov--McCarthy~\cite{IvanovMcCarthy} that there is a short exact sequence
\[1\to \Z/2\Z\to \Aut(\Mod_g^1)\to \Mod(\Sigma_g,\ast)\to 1,\]
where $\Mod(\Sigma_g,\ast)$ is the mapping class
group of a closed genus $g$ surface relative to a marked point.  Since every finite 
subgroup of $\Mod(\Sigma_g,\ast)$ is cyclic, every finite group of automorphisms of 
$\Mod_g^1$ is cyclic or dihedral; in particular, $\Sym_g$ cannot act faithfully 
on $\Mod_g^1$ for $g\gg 0$.
Even if we tried to work with closed surfaces, a faithful action of $\Sym_g$ 
on $\Mod_g$ would contradict Hurwitz's classical theorem that finite subgroups 
of $\Mod_g$ have size at most $84(g-1)$; see \cite[Theorem 7.4]{FarbMargalitPrimer}.
We will thus have to be content with constructing a weak FI-group $\Mod$ (this is our reason for introducing the notion of weak FI-groups).

\para{Systems of subsurfaces}
To pin down the morphisms in our weak FI-group, it will be helpful to realize the surfaces
supporting the various mapping class groups involved as subsurfaces of one infinite-genus 
surface.  Let
$S_{\N}$ be an infinite-genus surface with one end.  As in 
Figure~\ref{figure:bigsurface}a, pick closed subsurfaces $X_1,X_2,\ldots$, a
basepoint $\ast$, a ray $\alpha$, and arcs $\delta_1',\delta_2',\ldots$ with
the following properties.
\begin{compactitem}
\item The $X_i$ are disjoint and each is homeomorphic to a one-holed torus.
\item The subsurface
\[Y\coloneq S_{\N} \setminus \bigcup_{i=1}^{\infty} \Interior(X_i)\]
has genus $0$.
\item The ray $\alpha$ lies in $Y$ and starts at $\ast$.
\item The arc $\delta_i'$ lies in $Y$, starts at a point $p_i$ of $\alpha$, and ends at a point $\ast_i\in \partial X_i$. %the point $\ast_j\in \partial X_j$ is used in the proof of Proposition 4.5
Also, the arcs $\delta_i'$ are all disjoint from each other and their
interiors are disjoint from $\alpha$ and the $\partial X_j$.
\item The $p_i$ appear on $\alpha$ in their natural order and have no
accumulation points, and $p_1 = \ast$.
\end{compactitem}
Define $\delta_i$ to be the arc that starts at $\ast$, travels along $\alpha$ to $p_i$,
and then travels along $\delta_i'$.
For every finite set $I \subset \N$, let $S_I$ be a closed regular neighborhood
of $\bigcup_{i \in I} (\delta_i \cup X_i)$.
Observe that $S_I$ is a genus $\abs{I}$ surface with $1$ boundary component, and contains each handle $X_i$ for
$i \in I$; see Figure~\ref{figure:bigsurface}b.

\Figure{figure:bigsurface}{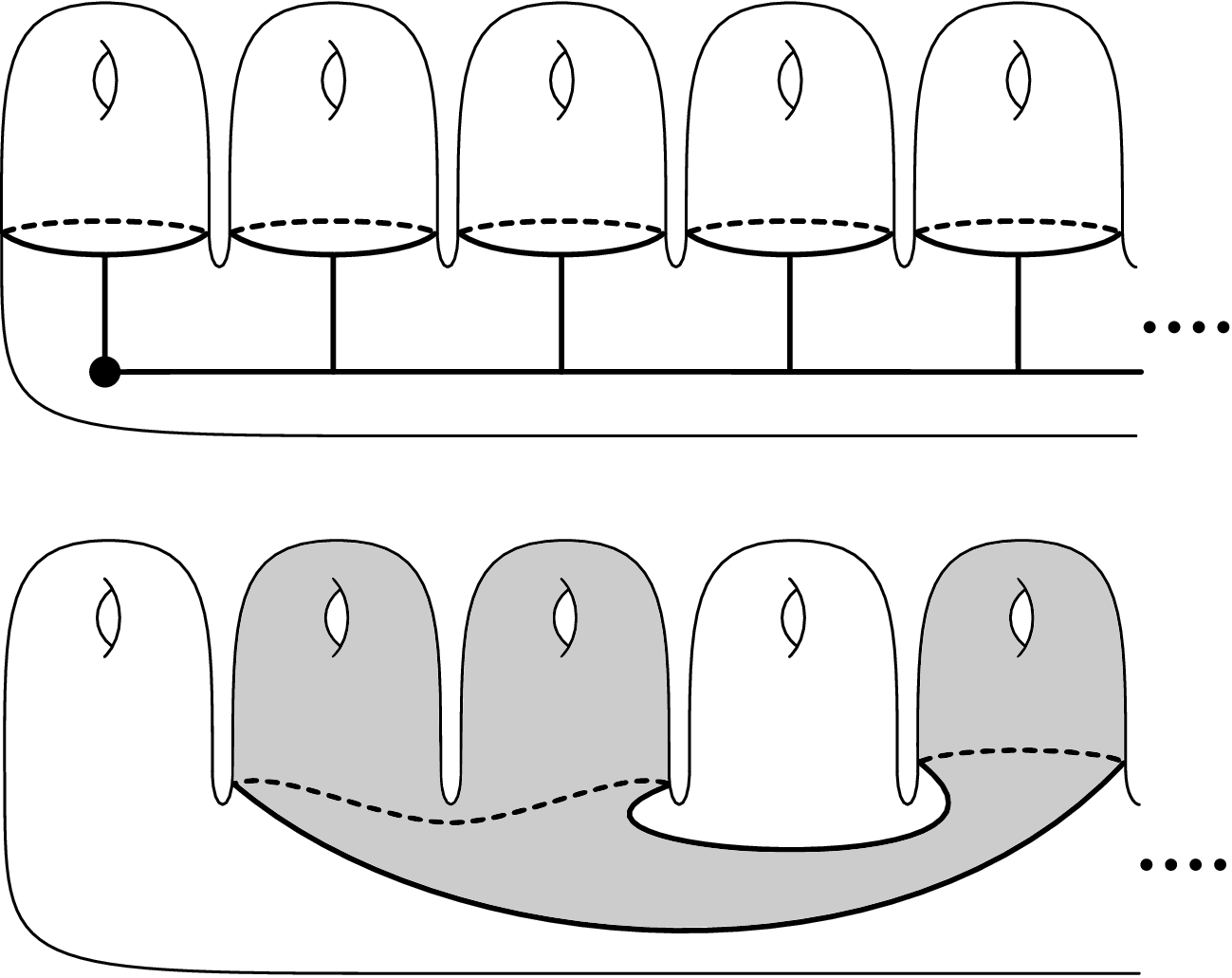}{The top shows the one-ended infinite-genus surface $S_\N$. The long ray shown is $\alpha$.  The bottom shows
the subsurface $S_{\{2,3,5\}}$ inside $S_\N$.  We have perturbed the subsurface by an isotopy to make its structure clear.}{85}

If $I,J \subset \N$ are finite sets such that $I \subset J$, then $S_I$ is isotopic
to a subsurface of $S_J$.  For our convenience,
we will assume that the $S_I$ are chosen so that in fact $S_I \subset S_J$ whenever $I \subset J$.  One way to achieve this is as follows.  Pick a Riemannian metric on
$S_{\N}$ such that for some $\epsilon > 0$, the closed neighborhood of radius
$\epsilon$ around $\bigcup_{i=1}^{\infty} (\delta_i \cup X_i)$ is a regular neighborhood
of $\bigcup_{i=1}^{\infty} (\delta_i \cup X_i)$.  Letting
$\eta\colon [0,\infty) \rightarrow (0,\epsilon)$ be a strictly increasing function, we define $S_I$ to be the closed neighborhood of radius $\eta(\abs{I})$ around
$\bigcup_{i \in I} (\delta_i \cup X_i)$.
We let $\iota_I^J\colon S_I\into S_J$ denote the inclusion, so $\iota_J^K\circ \iota_I^J=\iota_I^K$.

\para{The weak FI-group Mod}
We now define the weak FI-group $\Mod$.  For any surface $S$,
let $\Mod(S)$ denote the mapping class group of $S$, i.e.\ the group
of isotopy classes of orientation-preserving
homeomorphisms of $S$ that restrict to the identity on $\partial S$.
For each finite set $I \subset \N$, define $\Mod_I = \Mod(S_I)$.

We next define the distinguished homomorphisms $\Mod_I^J\colon \Mod_I\to \Mod_J$.
Consider a pair of finite sets $I\subset J \subset \N$.  By our assumption above 
we have $S_I \subset S_J$, so
we can define a homomorphism $\Mod_I^J\coloneq (\iota_I^J)_{\ast}\colon \Mod_I \rightarrow \Mod_J$ by extending mapping classes
on $S_I$ to $S_J$ by the identity.
Since $\iota_J^K\circ \iota_I^J=\iota_I^K$, these homomorphisms satisfy the compatibility
conditions of Definition~\ref{def:weakFIgroup}(iii).

We now define homomorphisms-modulo-conjugacy $\Mod_f\colon \Mod_I\to \Mod_J$ 
for each injection $f\colon I\into J$ between finite sets $I,J \subset \N$. 
Choose an arbitrary orientation-preserving embedding $S_I\to S_J$.  This induces
a homomorphism $\Mod(S_I)\to \Mod(S_J)$; we define 
$\Mod_f\colon \Mod_I\to \Mod_J$ to be the induced homomorphism-modulo-conjugacy.  
This definition may not seem very canonical, and we still need to check the compatibility conditions of Definition~\ref{def:weakFIgroup}(ii). This requires the following lemma.
\begin{lemma}
\label{lemma:modconjugate}
Let $S_0$ and $S$ be surfaces with one boundary component, and let 
$\phi,\phi'\colon S_0\into S$ be two orientation-preserving embeddings of $S_0$ into $S$.  
\begin{compactenum}[(i)]
\item The induced homomorphisms $\phi_{\ast},\phi'_{\ast}\colon \Mod(S_0)\into \Mod(S)$ are 
conjugate by an element of $\Mod(S)$.
\item If $\phi$ and $\phi'$ induce the same map $\HH_1(S_0;\Z)\to \HH_1(S;\Z)$ on homology, the homomorphisms $\phi_{\ast},\phi'_{\ast}\colon \Mod(S_0)\into \Mod(S)$ are conjugate by an element of $\Torelli(S)$. 
\end{compactenum}
\end{lemma}

\noindent
We prove Lemma \ref{lemma:modconjugate} below.  Part (i) of it shows that the homomorphism-modulo-conjugacy $\Mod_f$ does not depend on the choice
of embedding $S_I \into S_J$, so our definition was canonical after all. In particular, $\Mod_I^J$  represents $\Mod_{i_I^J}$. It also guarantees
that for all finite sets $I,J,K \subset \N$ and all injections 
$f\colon I \into J$ and
$g\colon J \into K$, we have $\Mod_{g \circ f}$ equal to 
$\Mod_g \circ \Mod_f$ in $\CGrp$, so the conditions of Definition~\ref{def:weakFIgroup}(ii)
are satisfied.  
This completes the construction of the weak FI-group $\Mod$.

\begin{proof}[Proof of Lemma~\ref{lemma:modconjugate}]
By perturbing $\phi$ and $\phi'$ by an isotopy, we can assume that their images lie in $\Interior(S)$. Let $T\coloneq S\setminus \Interior(\phi(S_0))$ and $T'\coloneq S\setminus \Interior(\phi'(S_0))$.
An Euler characteristic calculation shows that there exists an orientation-preserving
homeomorphism
$\psi_T\colon T\to T'$; moreover we may assume that $\psi_T$ agrees with $\phi'\circ \phi^{-1}$ on $\partial(\phi(S_0))=\partial T$.
Let $\psi \colon S \rightarrow S$ be the orientation-preserving homeomorphism that
restricts to $\phi' \circ \phi^{-1}$ on $\phi(S_0)$ and to $\psi_T$ on
$T$. We then have $\psi\circ \phi=\phi'$, so the mapping class defined by $\psi$
conjugates $\phi_{\ast}$ to $\phi'_{\ast}$, proving (i).

Let $V$ (resp.\ $V'$) be the image in $\HH_1(S;\Z)$ of $\HH_1(T;\Z)$ (resp.\ 
$\HH_1(T';\Z)$) under the map induced by the inclusion $T\into S$ (resp.\ $T'\into S$).
We have orthogonal decompositions $\HH_1(S;\Z)=\phi_{\ast}(\HH_1(S_0;\Z))\oplus V$ and $\HH_1(S;\Z)=\phi'_{\ast}(\HH_1(S_0;\Z))\oplus V'$.  If we assume as in (ii) that $\phi$ and $\phi'$ induce the same map $\HH_1(S_0;\Z)\to \HH_1(S;\Z)$, so that $\phi_{\ast}(\HH_1(S_0;\Z))=\phi'_{\ast}(\HH_1(S_0;\Z))$, it follows that the complementary subspaces $V$ and $V'$ are equal.
Recalling that $\psi_T$ is an orientation-preserving homeomorphism from $T$ to $T'$, the map $\psi_T$ induces a symplectic automorphism $M$ of
$V$.  We can realize $M$ by a homeomorphism $\zeta$ from $T'$ to itself 
(see \cite[Chapter 6]{FarbMargalitPrimer}).  Therefore replacing $\psi_T$ by $\zeta^{-1}\circ \psi_T$ 
in the previous paragraph, we may assume that $\psi_T$ acts trivially on $V$.  The assumption on 
$\phi$ and $\phi'$ means that $\phi'\circ \phi^{-1}$ acts trivially on $\HH_1(\phi(S_0);\Z)=\phi_{\ast}(\HH_1(S_0;\Z))$. 
It follows that $\psi\in \Torelli(S)$, proving (ii).
\end{proof}

\para{The weak FI-group $\Torelli$} 
We would like to define $\Torelli$ in the same way.  However, to ensure
that $\Torelli$ forms a weak FI-group we will need to be more careful with the
homomorphisms-modulo-conjugacy $\Torelli_f$.

For each $i\in \N$, fix once and for all  a symplectic basis $\{a_i,b_i\}$ for $\HH_1(X_i;\Z)$.
For any finite set $I \subset \N$, the map 
$\HH_1(S_I;\Z) \rightarrow \HH_1(S_{\N};\Z)$
is injective, and we will identify $\HH_1(S_I;\Z)$ with its image.
Therefore
$\Set{$a_i,b_i$}{$i \in \N$}$ is a symplectic basis for $\HH_1(S_{\N};\Z)$, and 
$\Set{$a_i,b_i$}{$i \in I$}$ is a symplectic basis for $\HH_1(S_I;\Z)$.

\begin{lemma}
\label{lemma:maponhomology}
For any injection $f\colon I \into J$ between finite sets $I,J \subset \N$, 
there exists an embedding $\phi_f\colon S_I \into S_J$ which on homology induces the map
\begin{equation}
\label{eq:homologymap}
\HH_1(S_I;\Z)\to \HH_1(S_J;\Z)\qquad\quad a_i\mapsto a_{f(i)},\quad b_i\mapsto b_{f(i)}\qquad\text{for all }i\in I.
\end{equation}
\end{lemma}
\begin{proof}
Let $\psi\colon S_I \rightarrow S_{f(I)}$ be an arbitrary orientation-preserving 
homeomorphism. Fix an arbitrary ordering on $I$. Then
$\Set{$\psi(a_i), \psi(b_i)$}{$i \in I$}$ and $\Set{$a_{f(i)},b_{f(i)}$}{$i \in I$}$
are both ordered symplectic bases for $\HH_1(S_{f(I)};\Z)$,  so there is a symplectic automorphism
$M$ of $\HH_1(S_{f(I)};\Z)$ taking the former to the latter.  We can realize $M$ 
by $\xi_f\in \Mod(S_{f(I)})$ \cite[Chapter 6]{FarbMargalitPrimer},  
and $\phi_f \coloneq i_{f(I)}^J\circ \xi_f \circ \psi$ is the desired map.
\end{proof}

We are now ready to define the weak FI-group $\Torelli$.  For each finite set $I\subset \N$, 
define $\Torelli_I$ to be the subgroup of $\Mod_I$ acting trivially on $\HH_1(S_I;\Z)$.

For each pair of finite sets $I \subset J\subset \N$, 
define $\Torelli_I^J\colon \Torelli_I\to \Torelli_J$ to be the restriction of 
the map $\Mod_I^J\colon \Mod_I\to \Mod_J$ described above.  The condition 
in Definition~\ref{def:weakFIgroup}(iii) is automatically satisfied.

For each injection $f\colon I \into J$ between finite sets $I,J \subset \N$, choose an arbitrary embedding $\phi_f\colon S_I \into S_J$ inducing the map \eqref{eq:homologymap} on homology, as guaranteed by
Lemma~\ref{lemma:maponhomology}.  We define $\Torelli_f\colon \Torelli_I\to \Torelli_J$ 
to be the restriction of the induced map $(\phi_f)_{\ast}\colon \Mod_I\to \Mod_J$.  By Lemma~\ref{lemma:modconjugate}(ii), any two embeddings inducing the map \eqref{eq:homologymap} on homology are $\Torelli_J$-conjugate, so this gives a well-defined homomorphism-modulo-conjugacy $\Torelli_f$. Moreover since the maps \eqref{eq:homologymap} are preserved under composition, 
Lemma~\ref{lemma:modconjugate}(ii) guarantees that these homomorphisms-modulo-conjugacy satisfy
the compatibility condition in Definition~\ref{def:weakFIgroup}(ii).  This 
concludes the construction of the weak FI-group $\Torelli$.

\para{The Johnson filtration}
For $k \geq 1$, we define a normal weak
FI-subgroup $\Torelli(k)$ of $\Torelli$ as follows.
For each finite set $I \subset \N$, choose a basepoint $\ast_I\in \partial S_I$
and let $\pi_1(S_I)\coloneq\pi_1(S_I,\ast_I)$.  We define $\Torelli(k)_I$ to
be the kernel of the action of $\Mod_I$ on $\pi_1(S_I) / \gamma_{k+1}(\pi_1(S_I))$. This kernel does not depend on the choice of basepoint. Note that $\Torelli(1)=\Torelli$. The following
lemma guarantees that $\Torelli(k)\normal \Torelli$.

\begin{lemma}
For any $k\geq 1$, if $f\colon I \into J$ is an injection between finite sets
$I,J \subset \N$, then $\Torelli_f(\Torelli(k)_I) \subset \Torelli(k)_J$.
\end{lemma}
\begin{proof}
Let $\lambda$ be an arc in $S_J \setminus \Interior(S_{f(I)})$ joining 
the basepoint $\ast_J\in\partial S_J$ to the basepoint $\ast_{f(I)}\in\partial S_{f(I)}$.  
There is an injection $\pi_1(S_{f(I)}) \into \pi_1(S_J)$ that takes 
$\delta \in \pi_1(S_{f(I)})$
to $\lambda \cdot \delta \cdot \lambda^{-1}$; we will identify $\pi_1(S_{f(I)})$ with its image
in $\pi_1(S_J)$.  The free group $\pi_1(S_J)$ can then be decomposed as a free
product $\pi_1(S_{f(I)}) \ast U$, where $U$ is a subgroup generated by loops
that lie entirely in $S_J \setminus \Interior(S_{f(I)})$.

Consider $\varphi \in \Torelli(k)_I$.
Since the embedding $\psi_f\colon S_I \into S_J$ induces a map taking
$\gamma_{k+1}(\pi_1(S_I))$ into 
$\gamma_{k+1}(\pi_1(S_{f(I)})) \subset \gamma_{k+1}(\pi_1(S_J))$,
the mapping class $\Torelli_f(\varphi)$ acts as the identity
on the image of $\pi_1(S_{f(I)})$ in $\pi_1(S_J) / \gamma_{k+1}(\pi_1(S_J))$.  
The mapping class
$\Torelli_f(\varphi)$ also acts as the identity on 
$S_J \setminus \Interior(S_{f(I)})$, and thus certainly acts
as the identity on the image of
$U$ in $\pi_1(S_J)/\gamma_{k+1}(\pi_1(S_J))$.  We conclude that $\Torelli_f(\varphi) \in \Torelli(k)_J$.
\end{proof}

\begin{proposition}
\label{prop:torellik}
$\{\Torelli(k)\}_{k=1}^\infty$ is a central filtration of 
$\Torelli=\Torelli(1)$ of finite rank.
\end{proposition}
\begin{proof}
Since $\gamma_2(\pi_1(S_I))\supset \gamma_3(\pi_1(S_I))\supset \gamma_4(\pi_1(S_I))\supset \cdots$, we have
\[\Torelli=\Torelli(1) \supset \Torelli(2) \supset \Torelli(3) \supset \cdots.\]
We must show for $k\geq 1$ that
$[\Torelli(1)_I,\Torelli(k)_I] \subset \Torelli(k+1)_I$ and
that $\Torelli(k)_I/\Torelli(k+1)_I$ is a finite-rank abelian group. 
Just as in Proposition~\ref{prop:IAk}, this is an immediate consequence of the higher Johnson homomorphisms for $\Torelli_g^{1}(k)$ (see \cite{SatohSurvey}).
\end{proof}

\subsection{Generating the Torelli group and its Johnson filtration}
\label{section:modgen}

Identify $\Sigma_g^1$ with $S_{[g]}$, so for all subsets
$I \subset \{1,\ldots,g\}$ we have a subsurface $S_I \subset \Sigma_g^1$.  As notation, if $S$
is a subsurface of $\Sigma_g^1$, we denote by $\Mod_g^1(S)$ the subgroup of $\Mod_g^1$ consisting
of mapping classes that are supported on $S$.  Also, define $\Torelli_g^1(S) \coloneq \Torelli_g^1 \cap \Mod_g^1(S)$.
The following result is a strengthening of the main result of the second author in \cite{PutmanSmallGenset}.

\begin{proposition}[Torelli generators]
\label{proposition:modgentorelli}
For $g \geq 3$, the group $\Torelli_g^1$ is generated by the subgroups
\[\Set{$\Torelli_g^1(S_I)$}{$I \subset \{1,\ldots,g\}$ satisfies $\abs{I} = 3$}.\]
\end{proposition}

Before proving Proposition~\ref{proposition:modgentorelli}, we deduce Theorem~\ref{maintheorem:smallgenset} from it.

\begin{proof}[{Proof of Theorem~\ref{maintheorem:smallgenset}}]
Johnson~\cite{JohnsonFinite} proved that $\Torelli_3^1$ is generated by $42$ elements. There are $\binom{g}{3}$  subsurfaces $S_I$
in Proposition~\ref{proposition:modgentorelli}, and each subgroup $\Torelli_g^1(S_I)$ is isomorphic to $\Torelli_3^1$,   so we deduce that $\Torelli_g^1$ is generated by $42 \binom{g}{3}$ elements.
There is a surjection $\Torelli_g^1 \twoheadrightarrow \Torelli_g$ obtained by gluing a disc to $\partial \Sigma_g^1$ and extending
mapping classes over the disc by the identity, so $\Torelli_g$ is also generated by $42 \binom{g}{3}$ elements.
\end{proof}

\begin{proof}[Proof of Proposition~\ref{proposition:modgentorelli}]
Let $\Gamma \subset \Torelli_g^1$ be the subgroup generated by the subgroups $\Torelli_g^1(S_I)$ for $\abs{I}=3$, or equivalently for $\abs{I}\leq 3$; our goal is to
prove that $\Gamma=\Torelli_g^1$.  We begin by describing some simple elements of $\Torelli_g^1$ that lie in $\Gamma$.

First, choose $i\in \{1,\ldots,g\}$. Recall the genus 1 subsurfaces $X_1,\ldots,X_g$ of $\Sigma_g^1$, which satisfy $X_i \subset S_I$ if and only if $i \in I$. The boundary curve $\partial X_i$ is a separating curve contained in $S_{\{i\}}$, so the Dehn twist $T_{\partial X_i}$ lies in $\Torelli_g^1(S_{\{i\}})\subset \Gamma$.

Next, choose $j\in \{1,\ldots,g\}$ with $j\neq i$, and let $\gamma$ be an embedded curve in $X_j$ based at $\ast_j\in \partial X_j$. The regular neighborhood of $\partial X_i\cup \delta_i\cup\delta_j\cup \gamma$ is a genus 0 surface with 3 boundary components. These 3 boundary components are isotopic to $\gamma$, the separating curve $\partial X_i$, and a third curve $\gamma'$ homologous to $\gamma$.  The mapping class $T_\gamma T_{\gamma'}^{-1}$ has the effect of ``sliding'' the handle $X_i$ around the curve $\delta_i^{-1}\delta_j\gamma\delta_j\delta_i^{-1}$ (though this notion is only well-defined modulo powers of $T_{\partial X_i}$); see \cite[Fact 4.7]{FarbMargalitPrimer}. Since $\gamma$ and $\gamma'$ are homologous, $T_\gamma T_{\gamma'}^{-1}$ lies in $\Torelli_g^1$. 
Since our regular neighborhood is contained in $S_{\{i,j\}}$, we have $T_\gamma T_{\gamma'}^{-1} \in \Torelli_g^1(S_{\{i,j\}})\subset \Gamma$.

For any subset
$I \subset \{1,\ldots,g\}$, define the subsurface
\[\YY_I \coloneq \Sigma_g^1 \setminus \big(\bigcup_{i \notin I} \Interior(X_i) \big),\]
so $\YY_I$ is a genus $\abs{I}$ surface with $g-\abs{I}+1$ boundary components. See Figure~\ref{figure:splittingsurface} for an example.
For $i\notin I$, let $\ZZ_I^{(i)}$ be the genus $\abs{I}$ surface with $g-\abs{I}$ boundary components obtained from $\YY_I$ by attaching a single disk to the boundary component $\partial X_i$. We will next show that the kernel of the corresponding map $\pi^{(i)}\colon \Torelli_g^1(\YY_I)\onto \Torelli_g^1(\ZZ_I^{(i)})$ is contained in $\Gamma$.
\Figure{figure:splittingsurface}{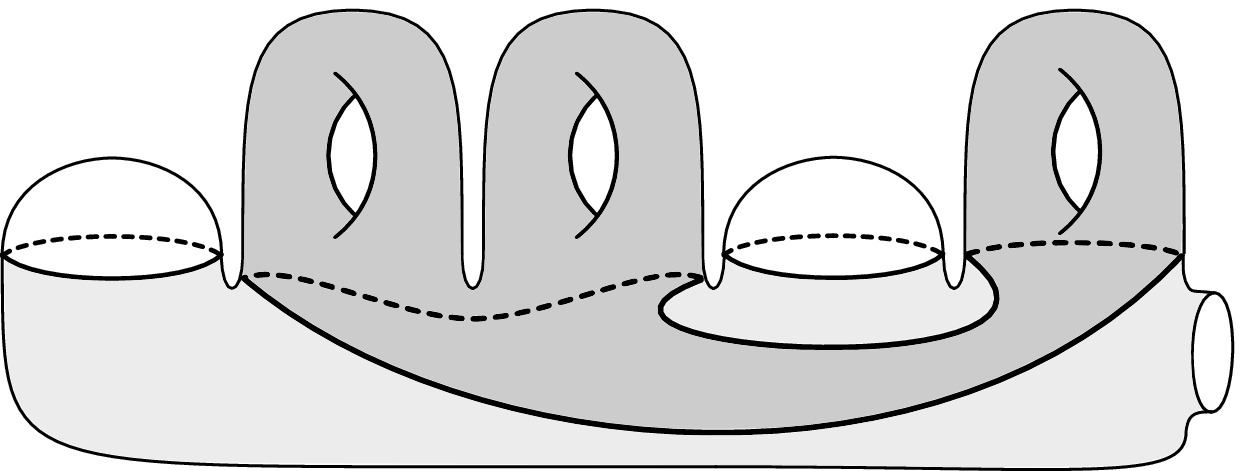}{For $I=\{2,3,5\}$, the surfaces $S_I$ (dark gray), $\YY_I$ (light and dark gray), and $\ZZ_I$.}{85}

Certainly $T_{\partial X_i}$ lies in $\ker(\pi^{(i)})$. Birman proved in \cite{BirmanSeq} that the quotient $\ker(\pi^{(i)})/\langle T_{\partial X_i}\rangle$ is isomorphic to $\pi_1(\ZZ_I^{(i)})$, with a loop in $\pi_1(\ZZ_I^{(i)})$ corresponding to the mapping class that slides the handle $X_i$ around that loop. The fundamental group $\pi_1(\ZZ_I^{(i)})$ can be generated by elements of the form $\delta_i^{-1}\delta_j\gamma\delta_j\delta_i^{-1}$ where $\gamma$ is an embedded curve in $X_j$: for each $j\not\in I$ we take $\gamma=\partial X_j$, and for each $j\in I$ we take two embedded curves generating $\pi_1(X_i)$. We saw earlier that  $\Gamma$ contains the mapping class $T_\gamma T_{\gamma'}^{-1}$ which slides the handle $X_i$ along any such loop, and so we conclude that $\ker(\pi^{(i)})\subset \Gamma$.

Let $\ZZ_I$ be the abstract surface obtained from $\YY_I$ by attaching disks to each of the boundary components $\partial X_i$ for $i\notin I$, so $\ZZ_I$ is a genus $\abs{I}$ surface with one boundary component. We can find an identification of $\ZZ_I$ with $S_I$ so that the composition $S_I\into \YY_I\into \ZZ_I\cong S_I$ is isotopic to the identity. It follows that the resulting homomorphism $\pi\colon \Torelli_g^1(\YY_I)\onto \Torelli_g^1(S_I)$ is  a split surjection, with section given by the inclusion %be careful
 $\Torelli_g^1(S_I)\into \Torelli_g^1(\YY_I)$.  It follows from the classical Fadell--Neuwirth exact sequences 
\cite{FadellNeuwirth} that the kernel $\ker(\pi)\subset \Torelli_g^1(\YY_I)$ is generated by the subgroups $\ker(\pi^{(i)})$ for all $i\notin I$ (in fact, $\ker(\pi)$ is isomorphic to the $(g-\abs{I})$-strand pure framed braid group on the surface $S_I$, though we will not use this directly). We conclude from the previous paragraph that $\ker(\pi)\subset \Gamma$.
 
When $\abs{I}=3$ we have $\Torelli_g^1(S_I)\subset \Gamma$ by definition, so $\Torelli_g^1(\YY_I)\subset \Gamma$ as well. The second author proved in \cite{PutmanSmallGenset} that $\Torelli_g^1$ is generated by the set
\[\Set{$\Torelli_g^1(\YY_I)$}{$I \subset \{1,\ldots,g\}$ satisfies $\abs{I} = 3$},\] so we conclude that $\Gamma=\Torelli_g^1$, as desired.
\end{proof}

We would like to conclude from Proposition~\ref{proposition:modgentorelli} that the weak FI-group $\Torelli$ is boundedly generated. However since $\Torelli$ is only a weak FI-group, this conclusion is not at all automatic (cf. Remark~\ref{remark:weakboundedgen}). To do this, we need the
following lemma. Along with Proposition~\ref{proposition:modgentorelli}, this lemma is the reason for our care in \S\ref{section:modfi} when defining the system of subsurfaces $S_I$.
\begin{lemma}
\label{lemma:allsame}
Let $J,J' \subset \N$ be finite sets such that $\abs{J}=\abs{J'}$.  There exists a
bijection $\sigma\colon J \rightarrow J'$ and an orientation-preserving homeomorphism
$\phi\colon S_J \rightarrow S_{J'}$ such that for all $I \subset J$ the subsurface
$\phi(S_I)$ of $S_{J'}$ is isotopic to the subsurface $S_{\sigma(I)}$.
\end{lemma}
\begin{proof}
Let $\sigma\colon J \rightarrow J'$ be the unique order-preserving bijection. Recall from \S\ref{section:modfi} the basepoint $\ast$, the genus 1 subsurfaces $X_1,X_2,\ldots$, and the arcs $\delta_1,\delta_2,\ldots$ used to define the surfaces $S_I$. Using
the standard ``change of coordinates principle'' (see \cite[\S 1.3.2]{FarbMargalitPrimer}),
there exists a homeomorphism $\phi\colon S_J \rightarrow S_{J'}$ with the following three 
properties.
\begin{compactitem}
\item $\phi(\ast) = \ast$.
\item For all $i \in J$, we have $\phi(X_i) = X_{\sigma(i)}$.
\item For all $i \in J$, we have $\phi(\delta_i) = \delta_{\sigma(i)}$.
\end{compactitem}
From the definition of the surface $S_I$ we see that $\phi$ has the desired properties.
\end{proof}
We emphasize that Lemma~\ref{lemma:allsame} depends in an essential way on the precise details of our construction of the surfaces $S_I$ (unlike Lemmas~\ref{lemma:modconjugate} and \ref{lemma:maponhomology} above, which were rather tautological). We are now ready to prove Theorem~\ref{maintheorem:modgentorelli}.

\begin{proof}[{Proof of Theorem~\ref{maintheorem:modgentorelli}}]
We begin by showing that the weak FI-group $\Torelli$ is boundedly generated in degree 3. Fix a finite set $J\subset \N$. If $\abs{J}\leq 3$ the condition \eqref{eq:boundedgencondition} is vacuous, so assume that $\abs{J}>3$. Taking $g\coloneq \abs{J}$,  let $\sigma\colon J\to [g]$ be the bijection given by Lemma~\ref{lemma:allsame}, and $\phi\colon S_J\to S_{[g]}=\Sigma_g^1$ the corresponding homeomorphism.

Consider $I\subset J$ with $\abs{I}=3$. By construction, $\phi$ takes $S_I$ to the subsurface $S_{\sigma(I)}$ of $S_{[g]}$. Therefore the isomorphism $\phi_{\ast}\colon \Torelli_J\to \Torelli_g^1$ takes the subgroup $\Torelli_J(I)$ supported on $S_I$ to the subgroup $\Torelli_g^1(\sigma(I))$ supported on $S_{\sigma(I)}$. Proposition~\ref{proposition:modgentorelli} states that $\Torelli_g^1$ is generated by the subgroups $\Torelli_g^1(S_{\sigma(I)})$. We conclude that $\Torelli_J$ is generated by the subgroups $\Torelli_J(I)$ for $I\subset J$ satisfying $\abs{I}=3$. Therefore \eqref{eq:boundedgencondition} is satisfied, and the weak FI-group $\Torelli$ is boundedly generated in degree 3.

Proposition~\ref{prop:torellik} states that $\{\Torelli(k)\}_{k=1}^\infty$ is a central filtration of bounded rank. Applying Theorem~\ref{maintheorem:boundedgeneration}, we conclude that for all $k\geq 1$, there exists $B_k\geq 0$ so that $\Torelli(k)\normal \Torelli$ is boundedly normally generated in degree $B_k$. 

Fix $g\geq 0$, and let us apply this conclusion to $\Torelli(k)_{[g]}=\Torelli_g^1(k)$. The bounded normal generation of $\Torelli(k)\normal \Torelli$ states that $\Torelli(k)_{[g]}$ is generated by the $\Torelli_g^1$-conjugates of its subgroups $\Torelli(k)_{[g]}(I)$ for those $I\subset [n]$ with $\abs{I}\leq B_k$. The subgroup $\Torelli(k)_{[g]}(I)$ consists of those elements of $\Torelli_g^1(k)$ supported on the genus $\abs{I}$ subsurface $S_I\subset \Sigma_g^1$, so its $\varphi$-conjugate consists of those elements of $\Torelli_g^1(k)$ supported on the subsurface $\varphi(S_I)$. If $\varphi\in \Torelli_g^1$, the subsurface $\varphi(S_I)$ is homologically standard. Therefore $\Torelli_g^1(k)$ is generated by elements of $\Torelli_g^1(k)$ supported on homologically standard subsurfaces of genus $\leq B_k$, as desired.
\end{proof}

\section{Mod-\texorpdfstring{$p$}{p} filtrations}
\label{section:level}

Fix a prime $p \geq 2$.  In this section we discuss the modifications that must be done to our 
proofs of Theorems~\ref{maintheorem:modgentorelli}
and \ref{maintheorem:autgentorelli} to obtain proofs of
Theorems~\ref{maintheorem:modgenlevel} and \ref{maintheorem:autgenlevel}.  Almost
everything goes through verbatim.
There are only two places where additional work is necessary.

 The first occurs in the proofs
of Propositions~\ref{prop:IAk} and \ref{prop:torellik}, where the higher
Johnson homomorphisms are invoked.  These should be replaced with the
higher mod-$p$ Johnson homomorphisms constructed by Cooper
in \cite{CooperThesis}.  The second place where a new idea is needed
is in the analogues of Propositions~\ref{proposition:autgentorelli}
and \ref{proposition:modgentorelli}, which give generators for
$\IA_n$ and $\Torelli_g^1$.  We need generators for the level $p$
congruence subgroups $\Aut(F_n,p)$ and $\Mod_g^1(p)$.  These are given
in Propositions~\ref{proposition:autgenlevel} and \ref{proposition:modgenlevel} below. Given these results, the proofs of Theorems~\ref{maintheorem:modgenlevel} and \ref{maintheorem:autgenlevel} parallel exactly the proofs of Theorems~\ref{maintheorem:modgentorelli}
and \ref{maintheorem:autgentorelli}.

\subsection{Generators for \texorpdfstring{$\Aut(F_n,p)$}{Aut(Fn,p)}}
For a subset $I \subset \{1,\ldots,n\}$, let $\Aut(F_n,I)$ consist of automorphisms supported on the splitting $F_n=F_I\ast F_{[n]-I}$, as defined
in \S\ref{section:autgen}.  Define $\Aut(F_n,p,I) = \Aut(F_n,p) \cap \Aut(F_n,I)$.

\begin{proposition}[{Generators for $\Aut(F_n,p)$}]
\label{proposition:autgenlevel}
For any $n \geq 0$ and any prime $p \geq 2$, the group $\Aut(F_n,p)$ is generated by the subgroups
\[\big\{\!\Aut(F_n,p,I)\,\,\big|\,\,I\subset \{1,\ldots,n\}\text{ satisfies }\ \abs{I}\leq 3\big\}.\]
\end{proposition}

For the proof of Proposition~\ref{proposition:autgenlevel}, we will need a generating set for the level-$p$ congruence
subgroup $\SL_n(\Z,p)$ of $\SL_n(\Z)$, which is the kernel of the natural map
$\SL_n(\Z) \twoheadrightarrow \SL_n(\Z/p)$.  Given $r\in \Z$ and $1 \leq i,j \leq n$,
let $\epsilon^n_{ij}(r)$ be the $n \times n$ matrix with $(i,j)$ entry equal to $r$ and all other entries equal to zero.
For $1\leq i<n$, let $\beta^n_i(r)$ be the $n\times n$ matrix with $(i,i)$ and $(i,i+1)$ entries equal to $r$, with
$(i+1,i)$ and $(i+1,i+1)$ entries equal to $-r$, and all other entries equal to zero.

Given $r\in \Z$ and $i\neq j$, let
\[\mathcal{E}^n_{ij}(r) \coloneq \One_n + \epsilon^n_{ij}(r)\in \SL_n(\Z)\]
be the $n\times n$ elementary matrix whose diagonal entries are $1$ and whose $(i,j)$ entry is $r$.
Similarly, given $r\in \Z$ and $1\leq i<n$, let $\mathcal{B}^n_i(r)\coloneq \One_n+\beta^n_i(r)$.
For instance,
\[\mathcal{B}^4_2(7) = \left(\begin{array}{rrrr} 
1 & 0  & 0  & 0 \\
0 & 8  & 7  & 0 \\
0 & -7 & -6 & 0 \\
0 & 0  & 0  & 1 \end{array}\right).\]
We then have the following theorem of Sury--Venkataramana.

\begin{theorem}[Sury--Venkataramana~\cite{SuryVenkataramana}]
\label{theorem:sllevelgen}
For $n \geq 3$ and $p \geq 2$, the group $\SL_n(\Z,p)$ is generated by the set
\[\Set{$\mathcal{E}^n_{ij}(p)$}{$1 \leq i,j \leq n$, $i \neq j$} \cup \Set{$\mathcal{B}^n_i(p)$}{$1 \leq i < n$}.\]
\end{theorem}

Let $\GL_n(\Z,p)$ be the level-$p$ congruence subgroup $\ker(\GL_n(\Z) \rightarrow \GL_n(\Z/p))$.  For
any $M\in \GL_n(\Z)$ we have $\det M=\pm 1$; moreover, if $M\in \GL_n(\Z,p)$, then $M\equiv \One_n\bmod{p}$
implies that $\det M\equiv 1\bmod{p}$.  For $p\geq 3$ these together imply that $\det M=1$, and so
$\GL_n(\Z,p)=\SL_n(\Z,p)$. However for $p=2$ we have an extension
\[1\to \SL_n(\Z,2)\to \GL_n(\Z,2)\overset{\det}{\longrightarrow} \{\pm 1\}\to 1.\]
Let $\mathcal{N}_1\in \GL_n(\Z)$ be the matrix obtained from the identity matrix by replacing the $1$ at position
$(1,1)$ with a $-1$.  Then $\mathcal{N}_1\in \GL_n(\Z,2)$ has $\det \mathcal{N}_1=-1$, and
$\GL_n(\Z,2)$ is generated by $\SL_n(\Z,2)$ together with $\mathcal{N}_1$.

\begin{proof}[Proof of Proposition~\ref{proposition:autgenlevel}]
Let $\Gamma \subset \Aut(F_n,p)$ be the subgroup generated by the purported generators $\Aut(F_n,p,I)$ with $\abs{I}\leq 3$;
our goal is to prove that $\Gamma=\Aut(F_n,p)$.  The map $\pi\colon\Aut(F_n)\to \GL_n(\Z)$ is known to
be surjective, so we have a short exact sequence
\[1 \longrightarrow \IA_n \longrightarrow \Aut(F_n,p) \overset{\pi}{\longrightarrow}\GL_n(\Z,p) \longrightarrow 1.\]
Since $\IA_n(I)\subset \Aut(F_n,p,I)$, Proposition~\ref{proposition:autgentorelli} implies that $\IA_n\subset \Gamma$. It is therefore enough to show that $\pi(\Gamma)$ is all of $\GL_n(\Z,p)$.

Define automorphisms $\widetilde{\mathcal{E}}^n_{ij}(p)\in \Aut(F_n,p)$ for $1 \leq i,j \leq n$ with $i \neq j$, automorphisms
$\widetilde{\mathcal{B}}_i^n(p)\in \Aut(F_n,p)$ for $1\leq i<n$, and the automorphism
$\widetilde{\mathcal{N}}_1\in\Aut(F_n,2)$ via the following formulas.
\begin{align*}
\widetilde{\mathcal{E}}^n_{ij}(p)(x_{\ell}) &= \begin{cases}
x_j x_i^p & \text{if $\ell = j$},\\
x_{\ell}  & \text{otherwise},\end{cases}\\
\widetilde{\mathcal{B}}_i^n(p)(x_{\ell}) &= \begin{cases}
x_i(x_i x_{i+1}^{-1})^p & \text{if $\ell = i$},\\
x_{i+1}(x_i x_{i+1}^{-1})^p & \text{if $\ell = i+1$},\\
x_{\ell} & \text{otherwise}, \end{cases}\\
\widetilde{\mathcal{N}}_1(x_{\ell}) &= \begin{cases}
x_1^{-1} & \text{if $\ell = 1$},\\
x_{\ell} & \text{otherwise}.\end{cases}
\end{align*}
The automorphism $\widetilde{\mathcal{E}}^n_{ij}(p)$ is supported on the splitting $\langle x_i,x_j\rangle\ast\langle x_\ell\,|\,\ell\neq i,j\rangle$,
so $\widetilde{\mathcal{E}}^n_{ij}(p)\in \Aut(F_n,p,{\{i,j\}})$.  Similarly,
$\widetilde{\mathcal{B}}^n_{i}(p)\in \Aut(F_n,p,{\{i,i+1\}})$ and $\widetilde{\mathcal{N}}_1\in \Aut(F_n,2,{\{1\}})$.
These elements are therefore contained in $\Gamma$.  Direct computations show
that $\pi(\widetilde{\mathcal{E}}^n_{ij}(p))=\mathcal{E}^n_{ij}(p)$ and
$\pi(\widetilde{\mathcal{B}}_i^n(p))=\mathcal{B}_i^n(p)$ and
$\pi(\widetilde{\mathcal{N}}_1) = \mathcal{N}_1$. By Theorem~\ref{theorem:sllevelgen}
these elements generate $\GL_n(\Z,p)$, so we conclude that $\pi(\Gamma)=\GL_n(\Z,p)$, as desired.
\end{proof}

\subsection{Generators for \texorpdfstring{$\Mod_g^1(p)$}{Mod(p)}}
Recall from \S\ref{section:modgen} that $\Sigma_g^1=S_{[g]}$, so for any $I\subset \{1,\ldots,g\}$ we have a  subsurface $S_I$ of $\Sigma_g^1$. For any subsurface $S$ of $\Sigma_g^1$,
we denote by $\Mod_g^1(p,S)$ the subgroup $\Mod_g^1(p,S)\coloneq \Mod_g^1(p)\cap \Mod(S)$ consisting of mapping classes supported on $S$.

\begin{proposition}[{Level-$p$ generators}]
\label{proposition:modgenlevel}
For $g \geq 3$ and $p \geq 2$, the group $\Mod_g^1(p)$ is generated by the set
\[\Set{$\Mod_g^1(p,S_I)$}{$I \subset \{1,\ldots,g\}$ satisfies $\abs{I} = 3$}.\]
\end{proposition}
The level-$p$ congruence
subgroup $\Sp_{2g}(\Z,p)$  is the kernel of the natural map
$\Sp_{2g}(\Z) \rightarrow \Sp_{2g}(\Z/p)$.
To prove Proposition~\ref{proposition:modgenlevel}, we will need a generating set for $\Sp_{2g}(\Z,p)$ analogous to Theorem~\ref{theorem:sllevelgen}.

For $g \geq 1$ let $\One_g$ and $\Zero_g$ be the $g \times g$ identity matrix and zero matrix, respectively.  Recall from \S\ref{section:autgen} that for any $r\in \Z$, we defined
$\epsilon^g_{ij}(r)$ as the $g \times g$ matrix with $(i,j)$ entry equal to $r$ and  zero otherwise, and $\beta_i^g(r)$ as the $g\times g$ matrix with $(i,i)$ and $(i,i+1)$ entries equal to $r$, with $(i+1,i)$ and $(i+1,i+1)$ entries equal to $-r$, and zero otherwise.
We define $s\epsilon^g_{ij}(r)$ to be the $g \times g$ matrix with $(i,j)$ and $(j,i)$ entries equal to $r$ and zero otherwise; when $i\neq j$ this is just $\epsilon^g_{ij}(r)+\epsilon^g_{ji}(r)$, while when $i=j$ we have $s\epsilon^g_{ii}(r)=\epsilon^g_{ii}(r)$.

We can now describe our generating set.  First, for $1 \leq i\leq j \leq g$, define:
\[\mathcal{X}^g_{ij}(r) \coloneq \One_{2g}+ \left(\begin{matrix} \Zero_g & \Zero_g \\ s\epsilon^g_{ij}(r)& \Zero_g \end{matrix}\right)
, \qquad
\mathcal{Y}^g_{ij}(r) \coloneq \One_{2g}+ \left(\begin{matrix} \Zero_g & s\epsilon^g_{ij}(r) \\ \Zero_g & \Zero_g \end{matrix}\right)\]
Second, for $1 \leq i,j \leq g$ with $i \neq j$, define:
\[\mathcal{Z}^g_{ij}(r) \coloneq \One_{2g}+\left(\begin{matrix}\epsilon^g_{ij}(r) & \Zero_g \\ \Zero_g & - \epsilon^g_{ji}(r) \end{matrix}\right)\]
Third, for $1 \leq i < g$, define:
\[\mathcal{W}^g_{i}(r) \coloneq \One_{2g}+ \left(\begin{matrix} \beta^g_{i}(r) & \Zero^g\\ \Zero_g & -\beta^g_{i}(r)^\top \end{matrix}\right)\]
Finally, define:
\[\mathcal{U}^g_{1}(r) \coloneq \One_{2g}+ \left(\begin{matrix} \epsilon^g_{11}(r) & \epsilon^g_{11}(r) \\ -\epsilon^g_{11}(r) & -\epsilon^g_{11}(r) \end{matrix}\right)\]

\begin{lemma}
\label{lemma:splevelgen}
For $g \geq 2$ and $p \geq 2$ the congruence group $\Sp_{2g}(\Z,p)$ is generated by the set
\[
\{\mathcal{X}^g_{ij}(p), \mathcal{Y}^g_{ij}(p)|1 \leq i\leq j \leq g\}
\cup
\{\mathcal{Z}^g_{ij}(p)|1 \leq i,j \leq g,\ i \neq j\}
\cup
\{\mathcal{W}^g_i(p)|1 \leq i < g\}
\cup
\{\mathcal{U}^g_1(p)\}
\]
\end{lemma}

\begin{proof}[{Proof of Lemma~\ref{lemma:splevelgen}}]
Let $\Gamma \subset \Sp_{2g}(\Z,p)$ be the subgroup generated by the purported generating set. Let $\SpLie_{2g}(\Z/p)$ and $\GLLie_{2g}(\Z/p)$ be the symplectic Lie algebra and matrix Lie algebra
over $\Z/p$, considered as abelian groups.

Let $\rho\colon \Sp_{2g}(\Z,p)\to \GLLie_{2g}(\Z/p)$ be the map sending  $\One_{2g}+pA\in \Sp_{2g}(\Z,p)$ to the mod-$p$ reduction of $A$ in $\GLLie_{2g}(\Z/p)$. It was first proved by Newman--Smart~\cite[Theorem 7]{NewmanSmartSp} that the image $\rho(\Sp_{2g}(\Z,p))$ is precisely the subgroup $\SpLie_{2g}(\Z/p)\subset \GLLie_{2g}(\Z/p)$.

It is
easy to see that $\rho$ maps our purported generating set to a basis for $\SpLie_{2g}(\Z/p)$ (the generators $\mathcal{W}^g_i(p)$ are needed
to get matrices whose diagonal does not vanish, and $\mathcal{U}_g^1(p)$ is needed to get matrices whose trace is not zero in each block). Therefore $\rho(\Gamma)=\SpLie_{2g}(\Z/p)=\rho(\Sp_{2g}(\Z,p))$.

It remains to show that $\ker(\rho)\subset\Gamma$. But the kernel of $\rho$ is $\Sp_{2g}(\Z,p^2)$ by definition, and Tits~\cite[Proposition 4]{TitsCongruence} proved
that $\Sp_{2g}(\Z,p^2) \subset \Gamma$ (the generator $\mathcal{U}_1^g(p)$ is not necessary here). Therefore $\Gamma=\Sp_{2g}(\Z,p)$, as desired.
\end{proof}

\begin{remark}
Bass--Milnor--Serre~\cite[Theorem 12.4]{BassMilnorSerre} proved that $\Sp_{2g}(\Z,p)$ is the normal closure in $\Sp_{2g}(\Z)$ of
\[\Set{$\mathcal{X}^g_{ij}(p), \mathcal{Y}^g_{ij}(p)$}{$1 \leq i\leq j \leq g$}\]
for $g \geq 2$ and $p \geq 2$.  However, one can show that these $g^2+g$ generators do not suffice to generate $\Sp_{2g}(\Z,p)$. Indeed, we saw above that $\Sp_{2g}(\Z,p)$ surjects to $\SpLie_{2g}(\Z/p)$, an elementary abelian group of rank $2g^2+g$, so $\Sp_{2g}(\Z,p)$ cannot be generated by fewer than $2g^2+g$ elements. Since the generating set in Lemma~\ref{lemma:splevelgen} consists of exactly $2g^2+g$ elements, it is in fact a \emph{minimal} generating set.
\end{remark}

\begin{proof}[Proof of Proposition~\ref{proposition:modgenlevel}]
Let $\Gamma \subset \Mod_g^1(p)$ be the subgroup generated by the subgroups $\Mod_g^1(p,S_I)$ with $\abs{I}=3$, or equivalently with $\abs{I}\leq 3$. We have a short exact sequence
\[1 \longrightarrow \Torelli_g^1 \longrightarrow \Mod_g^1(p) \xrightarrow{\pi} \Sp_{2g}(\Z,p) \longrightarrow 1.\]
Consider the image $\pi(\Gamma)\subset \Sp_{2g}(\Z,p)$. By examination we see that the generators $\mathcal{X}^g_{ij}(p)$,
$\mathcal{Y}^g_{ij}(p)$, and $\mathcal{Z}^g_{ij}(p)$ are in the image of $\Mod_g^1(p,S_{\{i,j\}})$,
the generator
$\mathcal{W}^g_i(p)$ is in the image of $\Mod_g^1(p,S_{\{i,i+1\}})$, and the generator $\mathcal{U}^g_1(p)$ is in the image of $\Mod_g^1(p,S_{\{1\}})$. By Lemma~\ref{lemma:splevelgen}, this shows that $\pi(\Gamma)=\Sp_{2g}(\Z,p)$.
Since $\Torelli_g^1(I)\subset \Mod_g^1(p,I)$, Proposition~\ref{proposition:modgentorelli} implies  that $\Torelli_g^1\subset\Gamma$. 
We conclude that $\Gamma=\Mod_g^1(p)$, as desired.
\end{proof}

\section{Lower bounds on generators}
\label{section:bounds}

Our goal now is to prove Theorems~\ref{maintheorem:modnongentorelli} and
\ref{maintheorem:autnongentorelli}.  We begin by recalling some facts about the
higher Johnson homomorphisms.  See Satoh~\cite{SatohSurvey}
for more details.

% if ref needed for injectivity: M. Lazard, "Sur les algbres enveloppantes universelles de certain algbres de Lie" C.R. Acad. Sci. Paris SŽr. I Math. , 234 (1952) pp. 788Ð791

\paragraph{Automorphism groups of free groups.}
Fix $n\geq 1$, and let $H\coloneq F_n^{\ab}$.  Since $\gamma_k(F_n)$ is a central filtration, the graded quotients $\gr_k(F_n)\coloneq \gamma_k(F_n)/\gamma_{k+1}(F_n)$ form a graded Lie algebra $\gr(F_n)$ under the commutator bracket. Witt~\cite{Witt} proved that $\gr(F_n)$ is naturally isomorphic to the free Lie algebra $\Lie(H)$ on $H=\gr_1(F_n)$.

Similarly, from the central filtration $\IA_n(k)$ we obtain a graded Lie algebra $\gr(\IA_n)$ with $\gr_k(\IA_n)\coloneq \IA_n(k)/\IA_n(k+1)$. The action of $\IA_n$ on $F_n$ induces an injective map of Lie algebras $\tau\colon \gr(\IA_n)\into \Der(\gr(F_n)) \cong \Der(\Lie(H))\cong\Hom(H,\Lie(H))$. Traditionally one thinks of the $k^{\text{th}}$ graded piece of $\tau$ as a homomorphism $\tau_k\colon \IA_n(k)\to \Hom(H,\Lie_{k+1}(H))$ with $\ker(\tau_k)=\IA_n(k+1)$; the map $\tau_k$ is known as the $k^{\text{th}}$ Johnson homomorphism. Explicitly, given $\varphi\in \IA_n(k)$ and $x\in F_n$ we have $\varphi(x)x^{-1}\in \gamma_{k+1}(F_n)$, and $\tau_k(\varphi)\in \Hom(H,\Lie_{k+1}(H))$ is the map that takes $[x]\in H$ to $[\varphi(x)x^{-1}]\in \gr_{k+1}(F_n)\cong\Lie_{k+1}(H)$.  Determining the image of $\tau$ is a fundamental and difficult problem which has a large literature (see
\cite{SatohSurvey} for a discussion; we especially would like to point out the papers Satoh~\cite{SatohLCS} and Enomoto--Satoh~\cite{EnomotoSatohDer}).  

The universal enveloping algebra of $\Lie(H)$ is the tensor algebra $T(H)$, that is, the free associative algebra on $H$. Since $\Lie(H)$ is a free $\Z$-module, the natural map $i\colon \Lie(H)\into T(H)$ to its universal enveloping algebra $T(H)$ is injective by the Poincar\'{e}--Birkhoff--Witt theorem (\cite[Theorem I.2.7.1]{BourbakiLie13}; see especially 
\cite[Corollary I.2.7.2]{BourbakiLie13}).

\begin{proof}[Proof of Theorem~\ref{maintheorem:autnongentorelli}]
Fix $k\geq 1$, and let $\rho\colon \Lie_{k+1}(H)\to H\otimes \bwedge^k H$ be the composition \[\rho\colon \Lie_{k+1}(H)\into H^{\otimes k+1}\onto H\otimes \bwedge^k H\] of the injection $i\colon \Lie_{k+1}(H)\into H^{\otimes k+1}$ with the natural projection. Denote by $\widehat{\tau}_k\colon \IA_n(k)\to \Hom(H,H\otimes \bwedge^k H)$ the composition
\[\widehat{\tau}_k\colon\IA_n(k)\xrightarrow{\tau_k}\Hom(H,\Lie_{k+1}(H))\rightarrow\Hom(H,H\otimes \bwedge^k H),\]
where the second map is induced by $\rho$.

Consider an automorphism $\varphi\in \IA_n(k)$ supported on the splitting $F_n=A\ast B$, and let $H_A\coloneq A^{\ab}\subset H$. From the explicit description of $\tau_k(\varphi)$ above, it is easy to see that $\tau_k(\varphi)$ lies in the subspace $\Hom(H_A,\Lie_{k+1}(H_A))\subset \Hom(H,\Lie_{k+1}(H))$. From the naturality of the Poincar\'{e}--Birkhoff--Witt injection, $\widehat{\tau}_k(\varphi)$ lies in $\Hom(H_A,H_A\otimes \bwedge^k H_A)$. If the splitting $F_n=A\ast B$ has rank $r<k$, then since $H_A\cong \Z^r$ we have $\bwedge^k H_A=0$, so $\widehat{\tau}_k(\varphi)=0$. This shows that any automorphism $\varphi\in \IA_n(k)$ supported on a splitting of rank less than $k$ has $\widehat{\tau}_k(\varphi)=0$.

To complete the proof of Theorem~\ref{maintheorem:autnongentorelli}, it thus suffices to show that $\widehat{\tau}_k(\IA_n(k))\neq 0$ when $n>k$. Since $F_n$ is centerless, conjugation gives an injection $\InnerAut\colon F_n\into \IA_n$. This corresponds under $\tau$ to the injection $\InnerDer\colon \Lie(H)\into \Der(\Lie(H))$:
\[\xymatrix{
F_n\ar[d]_{\InnerAut} & \gr(F_n) \ar^{\cong}[r] \ar[d]& \Lie(H) \ar[d]^{\InnerDer}\\
\IA_n & \gr(\IA_n) \ar[r]_{\tau} & \Der(\Lie(H))
}\]
Explicitly, the inner derivation corresponding to an element $\lambda\in \Lie_k(H)$ is the map $\eta_\lambda\in \Hom(H,\Lie_{k+1}(H))$ defined by $\eta_\lambda(h)=[\lambda,h]$ for $h\in H$.

Let $\{a_1,\ldots,a_n\}$ be a free basis for $H$, and set 
\[\lambda\coloneq [[\cdots[[a_1,a_2],a_3],\cdots],a_k]\in \Lie_k(H).\]
The commutativity of the  diagram above implies that all inner derivations lie in the image of $\tau$, so there exists some $\varphi\in \IA_n(k)$ with $\tau_k(\varphi)=\eta_\lambda$. It thus suffices to show that $\rho\circ \eta_\lambda\neq 0$; we do this by verifying that the element \[\widehat{\tau}_k(\varphi)(a_{k+1})=\rho(\eta_\lambda(a_{k+1}))=\rho([\lambda,a_{k+1}])\] is nonzero.

The image $i([\lambda,a_{k+1}])\in H^{\otimes k+1}$ is an alternating sum of $2^{k+1}$ monomials, each of the form $a_{\sigma(1)}\otimes \cdots\otimes a_{\sigma(k+1)}$ for some  permutation $\sigma\in \Sym_{k+1}$. However, by induction on $k$ we can see that the only such permutation $\sigma$ with $\sigma(1)=1$ is the identity $\id\in \Sym_{k+1}$. Accordingly, let $a_1^*\colon H\to \Z$ be the dual functional, and $(a_1^*\otimes \id)\colon H^{\otimes k+1}\to H^{\otimes k}$ be the map that applies this functional to the first factor. We then have $(a_1^*\otimes \id)\circ i([\lambda,a_{k+1}])=a_2\otimes \cdots\otimes a_{k+1}\in H^{\otimes k}$. This projects to $a_2\wedge \cdots\wedge a_{k+1}\in \bwedge^k H$ under the natural projection, so $(a_1^*\otimes \id)\circ \rho([\lambda,a_{k+1}])=a_2\wedge \cdots\wedge a_{k+1}\neq 0$. This shows that $\rho([\lambda,a_{k+1}])\neq 0\in H\otimes \bwedge^k H$, so $\rho\circ \eta_\lambda=\widehat{\tau}_k(\varphi)\neq 0$. This demonstrates that $\widehat{\tau}_k(\IA_n(k))\neq 0$ when $n>k$, and thus completes the proof of the theorem.
\end{proof}

\paragraph{Mapping class groups.}
We now turn to Theorem~\ref{maintheorem:modnongentorelli}, which requires introducing the higher Johnson
homomorphisms for the mapping class group.  Fix $g \geq 1$, and set $n=2g$.
Choosing an isomorphism $\pi_1(\Sigma_g^1,\ast)\cong F_{2g}$, we obtain an embedding of $\Torelli_g^1$ into $\IA_n$. The central filtration $\Torelli_g^1(k)$ is taken to the central filtration $\IA_n(k)$, so we obtain an embedding $\gr(\Torelli_g^1)\into \gr(\IA_n)$ of graded Lie algebras. Setting $H\coloneq F_{2g}^{\ab}\cong H_1(\Sigma_g^1;\Z)$, we obtain from this embedding the $k^{\text{th}}$ Johnson homomorphism $\tau_k\colon \Torelli_g^1(k)\to \Hom(H,\Lie_{k+1}(H))$.

\begin{proof}[Proof of Theorem~\ref{maintheorem:modnongentorelli}]
Just like
for $\IA_n$, we define $\widehat{\tau}_k\colon \Torelli_g^1(k)\to \Hom(H,H\otimes \bwedge^k H)$ via the formula $\widehat{\tau}_k(\varphi) = \rho \circ \tau_k(\varphi)$, where $\rho\colon \Lie_{k+1}(H)\to H\otimes \bwedge^k H$ is the same map as before.

Consider a subsurface $S\subset \Sigma_g^1$ such that $S\cong \Sigma_h^1$. Choose a disjoint subsurface 
$T\subset \Sigma_g^1$ with $T\cong \Sigma_{g-h}^1$. Fix a basepoint $\ast_S\in \partial S$ and an arc connecting $\ast_S$ 
to the basepoint $\ast\in \partial\Sigma_g^1$, and similarly for $\ast_T\in \partial T$. In the usual way, this determines 
inclusions $\pi_1(S,\ast_S) \into \pi_1(\Sigma_g^1,\ast)$ and $\pi_1(T,\ast_T)\into \pi_1(\Sigma_g^1,\ast)$.  By van 
Kampen's theorem, we have a splitting $F_{2g}\cong \pi_1(\Sigma_g^1,\ast)=\pi_1(S,\ast_S) \ast \pi_1(T,\ast_T)$.

If $\varphi\in \Torelli_g^1(k)$ is supported on the subsurface $S$, the induced automorphism of $F_{2g}$ preserves this 
splitting, which is of rank $2h$. If $2h<k$, our computation in the proof of Theorem~\ref{maintheorem:autnongentorelli} 
thus shows that $\widehat{\tau}_k(\varphi)=0$. Therefore $\widehat{\tau}_k$ vanishes on any element of $\Torelli_g^1(k)$ 
supported on a subsurface $\Sigma_h^1$ of genus less than $\frac{k}{2}$. To complete the proof, it thus suffices to prove 
that $\widehat{\tau}_k(\Torelli_g^1(k))\neq 0$ when $g>k$.

In the proof of Theorem~\ref{maintheorem:autnongentorelli}, we made use of the map $\InnerDer\colon \Lie(H)\to \Der(\Lie(H))$,
which is determined by $\InnerDer_1\colon H\to \Der_1(\Lie(H))$. The image of $\tau_1\colon \Torelli_g^1\to \Der_1(\Lie(H))$ 
does not contain $\InnerDer_1(H)$, but the work of Johnson~{JohnsonAbelian} shows that $\tau_1(\Torelli_g^1)$ does 
contain the image of another map $\PP_1\colon H\to \Der_1(\Lie(H))$, defined as follows.

Fix a symplectic basis $\{a_1,b_1,\ldots,a_g,b_g\}$ for $H$ and let $\omega\in \Lie_2(H)$ represent 
the algebraic intersection form $\ialg$ on $H$, so $\omega=\sum_{i=1}^g [a_i,b_i]$. Given $x\in H$, we define 
\[\PP_1(x)\coloneq\big[h\mapsto [x,h]+\ialg(h,x)\omega\big]\in  \Hom(H,\Lie_2(H))\cong \Der_1(\Lie(H)).\]
The map $\PP_1$ induces a map of Lie algebras $\PP\colon \Lie(H)\to \Der(\Lie(H))$.  We remark that $\PP$ is
not injective.  The initials ``PP'' stand for ``point-pushing'', since the image of this map
turns out to be the image under the Johnson homomorphism of the point-pushing subgroup of the mapping class group. However, neither of these facts will be necessary for our proof.

Let $L\subset H$ be the isotropic subspace $\langle a_1,\ldots,a_g\rangle$. For any $x,y \in L$ we have $\ialg(x,y) = 0$, so
\[\PP_1(x)(y) = [x,y]+\ialg(y,x)\omega=[x,y].\]
It follows by induction that for \emph{any} $\mu_1, \mu_2 \in \Lie(L)$ we have
\begin{equation}
\label{eqn:PPgeneral}
\PP(\mu_1)(\mu_2)=[\mu_1,\mu_2].
\end{equation}

Consider the element $\lambda\coloneq[[\cdots[a_1,a_2],\cdots],a_k]\in \Lie_k(L)\subset \Lie_k(H)$. The work of 
Johnson in \cite[\S 6]{JohnsonAbelian} shows that $\Image(\PP_1)\subset \tau_1(\Torelli_g^1)$; indeed, generators for 
$\Image(\PP_1)$ can be realized by genus $g-1$ bounding pairs that lie in the point-pushing subgroup. 
Since $\PP$ is a map of Lie algebras, it follows that there exists some 
$\varphi\in \Torelli_g^1(k)$ with $\tau_k(\varphi)=\PP(\lambda)$. 

As long as $g>k$ we can consider $a_{k+1}\in L$,  and from \eqref{eqn:PPgeneral} we have $\PP(\lambda)(a_{k+1})=[\lambda,a_{k+1}]$.  During the proof of Theorem~\ref{maintheorem:autnongentorelli}, we showed that $\rho([\lambda,a_{k+1}]) \neq 0$, so $\widehat{\tau}_k(\varphi)(a_{k+1}) = \rho(\PP(\lambda)(a_{k+1}))$ is nonzero. Thus $\widehat{\tau}_k(\Torelli_g^1)\neq 0$ when $g>k$, completing the proof.
\end{proof}

%\paragraph{Final remarks.}
%Our proofs of Theorems~\ref{maintheorem:autnongentorelli} and \ref{maintheorem:modnongentorelli} do not establish the corresponding facts for the Johnson filtrations of $\Out(F_n)$ and $\Mod_g$. In fact, the elements that we use can be taken to lie in the kernel of the projections $\Aut(F_n)\onto \Out(F_n)$ and $\Mod_g^1\onto \Mod_g$. It seems likely that these generalizations could be proven using similar ideas, but making use of deeper results on the images of the higher Johnson homomorphisms.

\noindent
\begin{tabular*}{\linewidth}[t]{@{}p{\widthof{E-mail: {\tt churchmath.stanford.edu}}+0.75in}@{}p{\linewidth - \widthof{E-mail: {\tt churchmath.stanford.edu}} - 0.75in}@{}}
{\raggedright
Thomas Church\\
Department of Mathematics\\
Stanford University\\
450 Serra Mall\\
Stanford, CA 94305\\
E-mail: \myemail{church@math.stanford.edu}}
&
{\raggedright
Andrew Putman\\
Department of Mathematics\\
Rice University, MS 136 \\
6100 Main St.\\
Houston, TX 77005\\
E-mail: \myemail{andyp@math.rice.edu}}
\end{tabular*}


\begin{thebibliography}{}
\begin{footnotesize}
\setlength{\itemsep}{1pt}

\bibitem{BassMilnorSerre}
H. Bass, J. Milnor\ and\ J.-P. Serre, Solution of the congruence subgroup problem for ${\rm SL}\sb{n}\,(n\geq 3)$ and ${\rm Sp}\sb{2n}\,(n\geq 2)$, \emph{Inst. Hautes \'Etudes Sci. Publ. Math.} No. 33 (1967), 59--137. 

\bibitem{BestvinaBuxMargalitIA}
M. Bestvina, K.-U. Bux\ and\ D. Margalit, Dimension of the Torelli group for ${\rm Out}(F\sb n)$, \emph{Invent. Math.} 170 (2007), no.~1, 1--32. Available at \arXiv{math/0603177}.

\bibitem{BirmanSeq}
J. S. Birman, Mapping class groups and their relationship to braid groups, \emph{Comm. Pure Appl. Math.} 22 (1969), 213--238. 

\bibitem{BourbakiLie13}
N. Bourbaki, {\it Lie groups and Lie algebras. Chapters 1--3}, translated from the French, reprint of the 1975 edition, Elements of Mathematics (Berlin), Springer, Berlin, 1989. 

\bibitem{ChurchOrbits}
T. Church, Orbits of curves under the Johnson kernel, to appear in \emph{Amer. J. Math.} Available at \arXiv{1108.4511}.

\bibitem{ChurchEllenbergFarbFI}
T. Church, J. Ellenberg,\ and\ B. Farb,  FI-modules: a new approach to stability for $S_n$-representations, preprint 2012,  \arXiv{1204.4533v2}.

\bibitem{ChurchEllenbergFarbNagpal}
T. Church, J. Ellenberg, B. Farb, and R. Nagpal, FI-modules over Noetherian rings, to appear in \emph{Geom. Topol.} Available at \arXiv{1210.1854}.

\bibitem{CooperThesis}
J. Cooper, Two mod-$p$ Johnson filtrations, in preparation.

\bibitem{DayPutmanGen}
M. Day\ and\ A. Putman, The complex of partial bases for $F_n$ and finite generation of the Torelli subgroup of $\Aut(F_n)$, to appear in \emph{Geom. Dedicata}. Available at \arXiv{1012.1914}.

\bibitem{EnomotoSatohDer}
N. Enomoto\ and\ T. Satoh, On the derivation algebra of the free Lie algebra and trace maps, \emph{Algebr. Geom. Topol.} {\bf 11} (2011), no.~5, 2861--2901. Available at \arXiv{1012.2169}.

\bibitem{FadellNeuwirth}
E. Fadell\ and\ L. Neuwirth, Configuration spaces, \emph{Math. Scand.} {\bf 10} (1962), 111-118. 

\bibitem{FarbMargalitPrimer}
B. Farb\ and\ D. Margalit, {\it A primer on mapping class groups}, Princeton Mathematical Series, 49, Princeton Univ. Press, Princeton, NJ, 2012.

\bibitem{GaroufalidisLevine}
S. Garoufalidis\ and\ J. Levine, Finite type $3$-manifold invariants and the structure of the Torelli group.~I, \emph{Invent. Math.} 131 (1998), no.~3, 541--594.

\bibitem{HatcherMargalit}
A. Hatcher\ and\ D. Margalit, Generating the Torelli group, \emph{Enseign. Math.} 58 (2012), 165--188. Available at \arXiv{1110.0876}.

\bibitem{IvanovMcCarthy}
N. V. Ivanov\ and\ J. D. McCarthy, On injective homomorphisms between Teichm\"uller modular groups. I, \emph{Invent. Math.} 135 (1999), no.~2, 425--486. 

\bibitem{JohnsonFirst}
D. L. Johnson, Homeomorphisms of a surface which act trivially on homology, \emph{Proc. Amer. Math. Soc.} 75 (1979), no.~1, 119--125. 

\bibitem{JohnsonAbelian}
D.~Johnson, {An abelian quotient of the mapping class group
  {$\mathcal{I}_{g}$}}, \emph{Math. Ann.} {249} (1980), no.~3, 225--242.

\bibitem{JohnsonConj}
D. Johnson, Conjugacy relations in subgroups of the mapping class group and a group-theoretic description of the Rochlin invariant, \emph{Math. Ann.} 249 (1980), no.~3, 243--263. 

\bibitem{JohnsonSurvey}
D. Johnson, A survey of the Torelli group, in {\it Low-dimensional topology (San Francisco, Calif., 1981)}, 165--179, Contemp. Math., 20 Amer. Math. Soc., Providence, RI.

\bibitem{JohnsonFinite}
D. Johnson, The structure of the Torelli group. I. A finite set of generators for ${\cal I}$, \emph{Ann. Math.} 118 (1983), no.~3, 423--442.

\bibitem{JohnsonII}
D. Johnson, The structure of the Torelli group. II. A characterization of the group generated by twists on bounding curves, \emph{Topology} 24 (1985), no.~2, 113--126.

\bibitem{JohnsonAbel}
D. Johnson, The structure of the Torelli group. III. The abelianization of $\mathcal{T}$, \emph{Topology} 24 (1985), no.~2, 127--144. 

\bibitem{LickorishTwists}
W. B. R. Lickorish. A finite set of generators for the homeotopy group of a 2-manifold. \emph{Proc. Cambridge Philos. Soc.} 60 (1964), 769--778.

\bibitem{MagnusGenerators}
W. Magnus, \"Uber $n$-dimensionale Gittertransformationen, \emph{Acta Math.} 64 (1935), no.~1, 353--367.

\bibitem{Matsumoto}
M. Matsumoto, Arithmetic mapping class groups, to appear in Park City Mathematics Series.

\bibitem{MoritaExtension}
S. Morita, The extension of Johnson's homomorphism from the Torelli group to the mapping class group, \emph{Invent. Math.} 111 (1993), no.~1, 197--224.

\bibitem{MumfordAbelianQuotients}
D. Mumford, Abelian quotients of the Teichm\"{u}ller modular group. \emph{J. Analyse Math.} 18 (1967), 227--244.

\bibitem{NewmanSmartSp}
M. Newman\ and\ J. R. Smart, Symplectic modulary groups, \emph{Acta Arith.} 9 (1964), 83--89. 

\bibitem{Perron}
B. Perron, Filtration de Johnson et groupe de Torelli modulo $p$, $p$ premier, \emph{C. R. Math. Acad. Sci. Paris} 346 (2008), no.~11-12, 667--670. 

\bibitem{PowellTorelli}
J. Powell, Two theorems on the mapping class group of a surface, \emph{Proc. Amer. Math. Soc.} 68 (1978), no.~3, 347--350.

\bibitem{PutmanCutPasteTorelli}
A. Putman, Cutting and pasting in the Torelli group, \emph{Geom. Topol.} 11 (2007), 829--865. Available at \arXiv{math/0608373}.

\bibitem{PutmanSmallGenset}
A. Putman, Small generating sets for the Torelli group, \emph{Geom. Topol.} 16 (2012), no.~1, 111--125. Available at \arXiv{1106.3294}.

\bibitem{PutmanRepStabilityCongruence}
A. Putman, Stability in the homology of congruence subgroups, preprint 2012, \arXiv{1201.4876v4}.

\bibitem{SatohLCS}
T. Satoh, On the lower central series of the IA-automorphism group of a free group, \emph{J. Pure Appl. Algebra} {\bf 216} (2012), no.~3, 709--717. Available at:\newline \url{http://www.math.kyoto-u.ac.jp/preprint/2009/26satoh.pdf}

\bibitem{SatohSurvey}
T. Satoh, A survey of the Johnson homomorphisms of the automorphism groups of free groups and related topics, preprint 2012, \arXiv{1204.0876v2}.

\bibitem{SerreLie}
J.-P. Serre, {\it Lie algebras and Lie groups}, second edition, Lecture Notes in Mathematics, 1500, Springer, Berlin, 1992.

\bibitem{Stallings}
J. Stallings, Homology and central series of groups, \emph{J. Algebra} 2 (1965), 170--181.

\bibitem{SuryVenkataramana}
B. Sury\ and\ T. N. Venkataramana, Generators for all principal congruence subgroups of ${\rm SL}(n,{\bf Z})$ with $n\geq 3$, \emph{Proc. Amer. Math. Soc.} 122 (1994), no.~2, 355--358.

\bibitem{TitsCongruence}
J. Tits, Syst\`emes g\'en\'erateurs de groupes de congruence, \emph{C. R. Acad. Sci. Paris S\'er. A-B} 283 (1976), no.~9, A{\rm i}, A693--A695. 

\bibitem{Witt} E. Witt, Treue Darstellungen Liescher Ringe, \emph{J. Reine Angew. Math.} 177 (1937), 152--210.

\bibitem{Zassenhaus}
H. Zassenhaus, Ein Verfahren, jeder endlichen $p$-Gruppe einen Lie-Ring mit der Charakteristik $p$ zuzuordnen, \emph{Abh. Math. Sem. Univ. Hamburg} 13 (1939), no.~1, 200--207.

\end{footnotesize}
\end{thebibliography}
\end{document}